\documentclass[a4paper,reqno,10pt]{amsart}
\usepackage[marginratio=1:1,height=692.3pt,width=440pt,tmargin=80pt]{geometry}
\usepackage{graphicx}
\usepackage{amsmath,latexsym,amssymb,amsthm,mathtools}
\usepackage[nospace,noadjust]{cite}
\usepackage{textcomp}
\usepackage{setspace}

\usepackage{xcolor}
\usepackage{url}
\usepackage{enumerate}

\newtheoremstyle{tightplain}{0pt}{0pt}{\itshape}{}{\bfseries}{.}{ }{}
\newtheoremstyle{tightdef}{0pt}{0pt}{\normalfont}{}{\bfseries}{.}{ }{}
\theoremstyle{tightplain}
\newtheorem{thm}{Theorem}[section]
\newtheorem{proposition}[thm]{Proposition}

\newtheorem{lemma}[thm]{Lemma}
\newtheorem{corollary}[thm]{Corollary}

\theoremstyle{tightdef}
\newtheorem{definition}[thm]{Definition}
\newtheorem{remark}[thm]{Remark}
\newtheorem{exam}[thm]{Example}

\renewcommand{\ker}{\operatorname{Ker}}
\newcommand{\Id}{\operatorname{Id}}

\renewcommand{\aa}{\mathfrak a}
\newcommand{\g}{\mathfrak g}
\newcommand{\kk}{\mathbf k}
\newcommand{\lt}{\lhd}
\newcommand{\rt}{\rhd}
\newcommand{\gl}{\mathfrak{gl}}
\newcommand{\Tot}{\mathrm{Tot}}
\newcommand{\lL}{\ell_{\lt}}
\newcommand{\lR}{\ell_{\rt}}
\newcommand{\rL}{r_{\lt}}
\newcommand{\rR}{r_{\rt}}
\newcommand{\pair}[2]{\langle #1,#2\rangle}

\title[Kupershmidt--Nijenhuis structures on pre-Leibniz algebras]{Kupershmidt--Nijenhuis structures on pre-Leibniz algebras}
\author{Sami Benabdelhafidh}
\address{University of Sfax, Faculty of Sciences of Sfax, BP 1171, 3038 Sfax, Tunisia}
\email{abdelhafidhsami41@gmail.com, samibenabdelhafidh95@gmail.com}


\usepackage{hyperref}
\hypersetup{
pdftitle={Kupershmidt-Nijenhuis structures on pre-Leibniz algebras},
pdfauthor={Sami Benabdelhafidh},
colorlinks,
citecolor=blue,
filecolor=black,
linkcolor=blue,
urlcolor=black
}

\mathtoolsset{showonlyrefs,showmanualtags}

\begin{document}
\noeqref{eq:Leibrep2,eq:ON2a,eq:PL2,eq:dn2,eq:dn3,eq:eq1b,eq:key2,eq:rn2a}

\begin{abstract}
In this paper, we first construct the dual of a representation of a pre-Leibniz algebra and study infinitesimal deformations of a pre-Leibniz algebra with a representation, which lead to Nijenhuis pairs, characterizing trivial deformations, and to dual-Nijenhuis pairs. Building on these notions, we introduce Kupershmidt-(dual-)Nijenhuis structures and show that they give rise to hierarchies of pairwise compatible Kupershmidt operators, the converse holding under an invertibility condition. Finally, we show that Rota--Baxter--Nijenhuis structures, $r$-$n$ structures and BN-structures all arise as particular cases, and that the first two correspond to each other on quadratic pre-Leibniz algebras.
\end{abstract}

\maketitle

\thispagestyle{empty}

\tableofcontents

\section{Introduction}\label{sec:intro}

Nijenhuis operators on Lie algebras play an important role in the deformation theory of algebras, where they generate trivial deformations \cite{NR}, and in the theory of integrable systems, where they appear under the name of hereditary operators \cite{FF} (see also \cite{Dor}). They have been studied on many algebraic structures, for instance on pre-Lie algebras \cite{WSBL}; the deformation theory of Nijenhuis algebra morphisms was developed in \cite{BenabdelhafidhNijenhuisMorphisms}. Poisson--Nijenhuis structures were introduced by Magri and Morosi in the study of completely integrable systems \cite{MM} (see also \cite{KSM,KR}). On a Lie algebra, the $r$-$n$ structures of \cite{RRH} are the infinitesimal counterparts of invariant Poisson--Nijenhuis structures on Lie groups.

Rota--Baxter operators were introduced by Baxter \cite{Ba} and have found many applications, for example in renormalization \cite{CK} and in the splitting of operations of operads \cite{BBGN}; see \cite{Guo} for a survey. On a Lie algebra equipped with a nondegenerate invariant symmetric bilinear form, skew-symmetric solutions of the classical Yang--Baxter equation correspond to Rota--Baxter operators of weight zero \cite{STS}. In \cite{Ku}, Kupershmidt gave an operator form of the classical Yang--Baxter equation for an arbitrary representation; the resulting maps are now called Kupershmidt operators, $O$-operators or relative Rota--Baxter operators. For the adjoint representation one recovers Rota--Baxter operators of weight zero, while for the coadjoint representation one recovers skew-symmetric $r$-matrices. Kupershmidt operators provide solutions of the classical Yang--Baxter equation in larger Lie algebras \cite{B1,BGN}. The interplay between Kupershmidt operators and Nijenhuis operators on a Lie algebra with a representation was studied by Hu, Liu and Sheng \cite{HLS}, whose Kupershmidt-(dual-)Nijenhuis structures contain $r$-$n$ structures as a special case and produce compatible Kupershmidt operators. Analogous structures were studied on associative algebras \cite{LBS}, on Malcev and Hom-Malcev algebras \cite{Mabrouk,HarrathiMabroukNcibSilvestrov}, on pre-Malcev algebras \cite{QRZ}, and on Leibniz algebras by Sun and Jing \cite{SJ}.

Leibniz algebras were introduced by Bloh \cite{Bloh} and rediscovered by Loday \cite{Lod,LP}. They appear in the study of Courant algebroids \cite{KW}, in deformation quantization \cite{DW} and in the integration of Leibniz algebras \cite{BW}; their deformation theory was studied in \cite{Bal,FM}. Pre-Leibniz algebras were introduced by Tang and Sheng \cite{ST} under the name of Leibniz-dendriform algebras, and studied by Das \cite{Das}, who developed their cohomology and deformation theory. They arise from relative Rota--Baxter operators on Leibniz algebras \cite{ST}, in the same way as dendriform algebras arise from Rota--Baxter operators and $O$-operators on associative algebras \cite{Aguiar,U}.

The purpose of this paper is to introduce Kupershmidt-(dual-)Nijenhuis structures on a pre-Leibniz algebra with a representation, by adapting ideas from \cite{HLS,SJ}. A representation of a pre-Leibniz algebra consists of four actions $\lL,\lR,\rL,\rR$ satisfying nine identities, and its dual is not obtained by dualizing them separately: the correct formula mixes them (Proposition~\ref{prop:dual}). Another new point is a proof of the hierarchy of Kupershmidt operators (Proposition~\ref{thm:hierarchy}) which avoids powers of Nijenhuis operators.

Our main results are the following. Nijenhuis pairs are exactly the pairs that generate trivial deformations of a pre-Leibniz algebra with a representation (Theorem~\ref{thm:trivial}), and $(N,S)$ is a Nijenhuis pair with respect to $V$ if and only if $(N,S^{t})$ is a dual-Nijenhuis pair with respect to the dual representation (Proposition~\ref{thm:dualpair}). For a Kupershmidt-(dual-)Nijenhuis structure $(K,S,N)$, the map $S$ is a Nijenhuis operator on the induced pre-Leibniz algebra $V_K$, with $(V_K)_S=V_{NK}$ (Proposition~\ref{thm:SNij}), and $K$ and $NK$ are compatible Kupershmidt operators (Proposition~\ref{prop:47}); more generally, the operators $N^kK$, $k\in\mathbb N$, form a hierarchy of pairwise compatible Kupershmidt operators (Proposition~\ref{thm:hierarchy}), and two invertible Kupershmidt operators $K_1,K_2$ are compatible if and only if $K_1K_2^{-1}$ is a Nijenhuis operator (Corollary~\ref{cor:invcompat}); conversely, two compatible Kupershmidt operators, one of which is invertible, give rise to Kupershmidt-dual-Nijenhuis structures (Proposition~\ref{prop:48}). Finally, Rota--Baxter--Nijenhuis (RBN) structures are Kupershmidt-Nijenhuis structures with respect to the regular representation (see the discussion after Definition~\ref{def:RBN}), $r$-$n$ structures and BN-structures (closed bilinear forms with a compatible Nijenhuis operator) are Kupershmidt-dual-Nijenhuis structures with respect to the dual regular representation (Theorems~\ref{thm:rnON} and~\ref{thm:BN}), and on a quadratic pre-Leibniz algebra RBN-structures and $r$-$n$ structures correspond to each other (Theorem~\ref{thm:RBNrn}).

The paper is organized as follows. In Section~\ref{sec:pre}, we recall pre-Leibniz algebras, their representations and Kupershmidt operators, and construct the dual representation. In Section~\ref{sec:NP}, we study infinitesimal deformations of a pre-Leibniz algebra with a representation, and introduce Nijenhuis pairs and dual-Nijenhuis pairs. In Section~\ref{sec:ON}, we introduce Kupershmidt-(dual-)Nijenhuis structures and study their properties. In Section~\ref{sec:hier}, we study compatible Kupershmidt operators, and show that Kupershmidt-(dual-)Nijenhuis structures give rise to hierarchies of pairwise compatible Kupershmidt operators. In Section~\ref{sec:rn}, we study RBN-structures, $r$-$n$ structures and BN-structures.

Throughout the paper, $\kk$ is a field of characteristic $0$ and all vector spaces are finite dimensional over $\kk$.

\section{Preliminaries on pre-Leibniz algebras}\label{sec:pre}

In this section, we recall Leibniz and pre-Leibniz algebras, their representations \cite{Lod,LP,ST,Das} and Kupershmidt operators, and construct the dual representation.

\begin{definition}\label{def:Leib}
A (left) \emph{Leibniz algebra} $(\g,[\cdot,\cdot])$ is a vector space $\g$ equipped with a bilinear product $[\cdot,\cdot]:\g\otimes\g\to\g$ satisfying
\begin{equation}\label{eq:Leib}
[x,[y,z]]=[[x,y],z]+[y,[x,z]],\qquad \forall\, x,y,z\in\g .
\end{equation}
\end{definition}

\begin{definition}\label{def:Leibrep}
Let $(\g,[\cdot,\cdot])$ be a Leibniz algebra and $V$ a vector space. Let $\rho^L,\rho^R:\g\to\gl(V)$ be two linear maps satisfying, for all $x,y\in\g$,
\begin{align}
\rho^L([x,y])&=[\rho^L(x),\rho^L(y)],\label{eq:Leibrep}\\
\rho^R([x,y])&=[\rho^L(x),\rho^R(y)],\label{eq:Leibrep2}\\
\rho^R(y)\rho^L(x)&=-\rho^R(y)\rho^R(x).\label{eq:Leibrep3}
\end{align}
Then the triple $(V,\rho^L,\rho^R)$ is called a \emph{representation} of the Leibniz algebra $(\g,[\cdot,\cdot])$ \cite{LP,ST}.
\end{definition}

Equivalently, $\g\oplus V$ is a Leibniz algebra under the bracket $[x+u,y+v]=[x,y]+\rho^L(x)v+\rho^R(y)u$. For a linear map $A:V\to W$ we denote by $A^{t}:W^*\to V^*$ its transpose (dual map), $\pair{A^{t}\xi}{v}=\pair{\xi}{Av}$, and for a linear map $\rho$ with values in $\gl(V)$ we put $\rho^*(x):=-\rho(x)^{t}$, that is, $\pair{\rho^*(x)\xi}{v}=-\pair{\xi}{\rho(x)v}$. Then the triple $(V^*,(\rho^L)^*,-(\rho^L)^*-(\rho^R)^*)$ is a representation of $\g$, called the \emph{dual representation} \cite{ST}.

\begin{definition}[{\cite{ST,Das}}]
\label{def:preLeib}
A \emph{pre-Leibniz algebra}  is a triple $(\aa,\lt,\rt)$ consisting of a vector space $\aa$ and two bilinear products such that, for all $x,y,z\in\aa$,
\begin{align}
x\lt(y\lt z+y\rt z)&=(x\lt y)\lt z+y\rt(x\lt z),\label{eq:PL1}\\
x\rt(y\lt z)&=(x\rt y)\lt z+y\lt(x\lt z+x\rt z),\label{eq:PL2}\\
x\rt(y\rt z)&=(x\lt y+x\rt y)\rt z+y\rt(x\rt z).\label{eq:PL3}
\end{align}
A \emph{morphism} of pre-Leibniz algebras is a linear map preserving both products.
\end{definition}

Pre-Leibniz algebras split Leibniz algebras in the following sense \cite{ST,Das}: if $(\aa,\lt,\rt)$ is a pre-Leibniz algebra, then $x\circ y:=x\lt y+x\rt y$ is a Leibniz bracket on $\aa$; the Leibniz algebra $\aa_\Tot=(\aa,\circ)$ is called the \emph{total Leibniz algebra}. Conversely, if $R$ is a Rota--Baxter operator on a Leibniz algebra $\g$, i.e.\ $[Rx,Ry]=R([Rx,y]+[x,Ry])$, then $x\lt y:=[x,Ry]$ and $x\rt y:=[Rx,y]$ define a pre-Leibniz algebra on $\g$.

As for Leibniz algebras, representations of a pre-Leibniz algebra are defined through semidirect products \cite{Das}.

\begin{definition}\label{def:preLeibrep}
Let $(\aa,\lt,\rt)$ be a pre-Leibniz algebra. A vector space $V$ together with four linear maps $\lL,\lR,\rL,\rR:\aa\to\gl(V)$ such that $\aa\oplus V$, with the products
\begin{equation}
(x,u)\lt(y,v)=\big(x\lt y,\ \lL(x)v+\rL(y)u\big),\qquad (x,u)\rt(y,v)=\big(x\rt y,\ \lR(x)v+\rR(y)u\big),
\end{equation}
is a pre-Leibniz algebra, is called a \emph{representation} of $\aa$. This pre-Leibniz algebra is called the \emph{semidirect product} of $\aa$ and $V$, denoted by $\aa\ltimes V$.
\end{definition}

Thus $x\lt v=\lL(x)v$, $x\rt v=\lR(x)v$, $v\lt x=\rL(x)v$ and $v\rt x=\rR(x)v$. We put $\ell:=\lL+\lR$ and $r:=\rL+\rR$. Comparing the $V$-components of \eqref{eq:PL1}--\eqref{eq:PL3} in $\aa\ltimes V$, with exactly one argument in $V$, shows that $(V,\lL,\lR,\rL,\rR)$ is a representation if and only if, for all $x,y\in\aa$ (cf.\ \cite[Def.~4.1]{Das}),
\begin{align}
 \lL(x)\ell(y)&=\lL(x\lt y)+\lR(y)\lL(x),\tag{R1}\\\displaybreak[0]
 \lL(x)r(y)&=\rL(y)\lL(x)+\rR(x\lt y),\tag{R2}\\\displaybreak[0]
 \rL(x\circ y)&=\rL(y)\rL(x)+\lR(x)\rL(y),\tag{R3}\\\displaybreak[0]
 \lR(x)\lL(y)&=\lL(x\rt y)+\lL(y)\ell(x),\tag{R4}\\\displaybreak[0]
 \lR(x)\rL(y)&=\rL(y)\lR(x)+\rL(x\circ y),\tag{R5}\\\displaybreak[0]
 \rR(x\lt y)&=\rL(y)\rR(x)+\lL(x)r(y),\tag{R6}\\\displaybreak[0]
 \lR(x)\lR(y)&=\lR(x\circ y)+\lR(y)\lR(x),\tag{R7}\\\displaybreak[0]
 \lR(x)\rR(y)&=\rR(y)\ell(x)+\rR(x\rt y),\tag{R8}\\\displaybreak[0]
 \rR(x\rt y)&=\rR(y)r(x)+\lR(x)\rR(y).\tag{R9}
\end{align}

Let $L_\lt,R_\lt,L_\rt,R_\rt$ be the left and right multiplication operators, $L_\lt(x)y=R_\lt(y)x=x\lt y$ and $L_\rt(x)y=R_\rt(y)x=x\rt y$. Then $(\aa,L_\lt,L_\rt,R_\lt,R_\rt)$ is a representation of $\aa$, called the \emph{regular} (or adjoint) \emph{representation}.

The total Leibniz algebra of $\aa\ltimes V$ is $\aa_\Tot\ltimes V_\Tot$, where $V_\Tot:=(V,\ell,r)$, so $V_\Tot$ is a representation of $\aa_\Tot$. Moreover, substituting (R3) into (R5), (R6) into (R2) and (R9) into (R8), we get, for all $x,y\in\aa$,
\begin{align}
\rL(y)\big(\lR(x)+\rL(x)\big)&=0,\tag{E1}\label{eq:extra}\\
\rL(y)\big(\lL(x)+\rR(x)\big)&=0,\tag{E2}\label{eq:extra2}\\
\rR(y)\big(\ell(x)+r(x)\big)&=0 .\tag{E3}\label{eq:extra3}
\end{align}

We now construct the dual representation. It is not obtained by dualizing each of the four actions separately. As above, for $\zeta:\aa\to\gl(V)$ we put $\zeta^*(x):=-\zeta(x)^{t}$. Throughout the paper, the superscript ${}^{t}$ always denotes the transpose of a linear map (for instance $N^{t}$, $S^{t}$), while ${}^{*}$ applied to an action always includes the sign, as in $\zeta^*$; the superscript ${}^*$ on a vector space denotes its dual. Note that $(\zeta_1(x)\zeta_2(y))^{t}=\zeta_2^*(y)\zeta_1^*(x)$.

\begin{proposition}\label{prop:dual}
Let $(V,\lL,\lR,\rL,\rR)$ be a representation of a pre-Leibniz algebra $(\aa,\lt,\rt)$. Then
\begin{equation}\label{eq:dual}
V^\vee:=\big(V^*,\ -\lL^*,\ \lL^*+\lR^*,\ -\ell^*-r^*,\ \lL^*+\rR^*\big)
\end{equation}
is a representation of $\aa$, called the \emph{dual representation}. Moreover, $(V^\vee)^\vee=V$.
\end{proposition}
\begin{proof}
Put $a=\lL^*$, $b=\lR^*$, $c=\rL^*$, $d=\rR^*$, so that $\ell^*=a+b$ and $r^*=c+d$. Then the actions of \eqref{eq:dual} are $\lL'=-a$, $\lR'=\ell^*$, $\rL'=-\ell^*-r^*$, $\rR'=a+d$, and we have $\ell'=b$ and $r'=-b-c$. Dualizing (R1), (R2), (R4), (R8), \eqref{eq:extra2}, \eqref{eq:extra3}, and \eqref{eq:Leibrep}--\eqref{eq:Leibrep3} for $V_\Tot$, we obtain
\begin{align}
a(x\lt y)&=a(x)b(y)-\ell^*(y)a(x),\tag{D1}\\
d(x\lt y)&=a(x)c(y)-r^*(y)a(x),\tag{D2}\\
a(x\rt y)&=\ell^*(x)a(y)-a(y)b(x),\tag{D4}\\
d(x\rt y)&=\ell^*(x)d(y)-d(y)b(x),\tag{D8}\\
(a+d)(y)c(x)&=0,\tag{X2}\\
(\ell^*+r^*)(y)d(x)&=0,\tag{X3}\\
\ell^*(x\circ y)&=[\ell^*(x),\ell^*(y)],\tag{T1}\\
r^*(x\circ y)&=\ell^*(x)r^*(y)-r^*(y)\ell^*(x),\tag{T2}\\
(\ell^*+r^*)(x)r^*(y)&=0,\qquad\forall\,x,y\in\aa .\tag{T3}
\end{align}
Denote by (R1$'$)--(R9$'$) the identities (R1)--(R9) for \eqref{eq:dual}. For instance, by (D1) and (D2), we have
\begin{equation}
\begin{aligned}
&\lL'(x)\ell'(y)-\lL'(x\lt y)-\lR'(y)\lL'(x)=-a(x)b(y)+a(x\lt y)+\ell^*(y)a(x)=0,\\
&\lL'(x)r'(y)-\rL'(y)\lL'(x)-\rR'(x\lt y)=a(x)\big(b(y)+c(y)\big)-(\ell^*+r^*)(y)a(x)-(a+d)(x\lt y)\\
&\qquad=\big(a(x)b(y)-\ell^*(y)a(x)-a(x\lt y)\big)+\big(a(x)c(y)-r^*(y)a(x)-d(x\lt y)\big)=0,
\end{aligned}
\end{equation}
which are (R1$'$) and (R2$'$). Similarly, (R7$'$) is (T1), (R5$'$) is the sum of (T1) and (T2), and (R3$'$) reduces to (T3) by (R5$'$). Moreover, (R1$'$) and (R4$'$) are (D1) and (D4), (R2$'$) is the sum of (D1) and (D2), and (R6$'$) follows from (R2$'$) and (X3). Similarly, (R8$'$) is the sum of (D4) and (D8), and (R9$'$) follows from (R8$'$) and (X2). Thus $V^\vee$ is a representation of $\aa$. Finally, since $\zeta^{**}=\zeta$, the actions of $(V^\vee)^\vee$ are $\lL$, $-\lL+\ell=\lR$, $-\lR+(\lR+\rL)=\rL$ and $-\lL+(\lL+\rR)=\rR$, which implies that $(V^\vee)^\vee=V$.
\end{proof}

In particular, the dual of the regular representation is
\begin{equation}\label{eq:dualreg}
\aa^\vee=\big(\aa^*,\ -L_\lt^*,\ L^*_\circ,\ -L_\circ^*-R_\circ^*,\ L_\lt^*+R_\rt^*\big),
\end{equation}
where $L_\circ=L_\lt+L_\rt$ and $R_\circ=R_\lt+R_\rt$. It is called the \emph{dual regular representation}.

Let $(V,\lL,\lR,\rL,\rR)$ be a representation of a pre-Leibniz algebra $(\aa,\lt,\rt)$. To endow $V$ with a pre-Leibniz algebra structure, we need the notion of a Kupershmidt operator \cite{Ku}. A linear map $K:V\to\aa$ is called a \emph{Kupershmidt operator} (also called an $O$-operator or a relative Rota--Baxter operator) on $\aa$ with respect to $(V,\lL,\lR,\rL,\rR)$ if, for all $u,v\in V$,
\begin{equation}\label{eq:O}
\begin{aligned}
Ku\lt Kv&=K\big(\lL(Ku)v+\rL(Kv)u\big),\\
Ku\rt Kv&=K\big(\lR(Ku)v+\rR(Kv)u\big).
\end{aligned}
\end{equation}
If one writes, for any linear map $K:V\to\aa$,
\begin{equation}
u\lt_Kv:=\lL(Ku)v+\rL(Kv)u,\qquad u\rt_Kv:=\lR(Ku)v+\rR(Kv)u ,
\end{equation}
then \eqref{eq:O} reads $Ku\lt Kv=K(u\lt_Kv)$ and $Ku\rt Kv=K(u\rt_Kv)$. Since $(Ku,u)\lt(Kv,v)=(Ku\lt Kv,u\lt_Kv)$ and $(Ku,u)\rt(Kv,v)=(Ku\rt Kv,u\rt_Kv)$ in $\aa\ltimes V$, the map $K$ is a Kupershmidt operator if and only if its graph $G_K=\{(Ku,u)\mid u\in V\}$ is a subalgebra of $\aa\ltimes V$. In that case $u\mapsto(Ku,u)$ is an isomorphism from $(V,\lt_K,\rt_K)$ onto $G_K$, so $V$ becomes a pre-Leibniz algebra, denoted by $V_K=(V,\lt_K,\rt_K)$, and $K:V_K\to\aa$ is a morphism.
A Kupershmidt operator on $\aa$ with respect to the regular representation is called a \emph{Rota--Baxter operator}: $Rx\lt Ry=R(Rx\lt y+x\lt Ry)$ and $Rx\rt Ry=R(Rx\rt y+x\rt Ry)$.

Now consider the dual regular representation $\aa^\vee$. For $\pi\in\aa\otimes\aa$, define $\pi^\sharp:\aa^*\to\aa$ by
\begin{equation}
\pair{\pi^\sharp(\alpha)}{\beta}=\pi(\alpha,\beta),\qquad\forall\,\alpha,\beta\in\aa^*.
\end{equation}
We say that $\pi$ is a \emph{pre-Leibniz $r$-matrix} if $\pi^\sharp$ is a Kupershmidt operator on $\aa$ with respect to $\aa^\vee$; explicitly, for all $\alpha,\beta\in\aa^*$,
\begin{align}
\pi^\sharp\alpha\lt\pi^\sharp\beta&=\pi^\sharp\big(-L_\lt^*(\pi^\sharp\alpha)\beta-(L_\circ^*+R_\circ^*)(\pi^\sharp\beta)\alpha\big),\notag\\
\pi^\sharp\alpha\rt\pi^\sharp\beta&=\pi^\sharp\big(L_\circ^*(\pi^\sharp\alpha)\beta+(L_\lt^*+R_\rt^*)(\pi^\sharp\beta)\alpha\big).
\end{align}
Consequently, a pre-Leibniz $r$-matrix induces a pre-Leibniz algebra structure on $\aa^*$. Unlike \cite{ST}, we impose no symmetry condition on $\pi$, and do not consider a tensor form of these equations.

We conclude this section with an example in which the pre-Leibniz identities are not trivially satisfied. It will be used throughout the paper.

\begin{exam}\label{ex:3dim}
Let $\g$ be the Leibniz algebra with basis $\{e_1,e_2,e_3\}$ and nonzero brackets $[e_1,e_2]=e_2$, $[e_1,e_3]=e_3$, i.e.\ $[x,y]=x_1(y_2e_2+y_3e_3)$; both sides of \eqref{eq:Leib} equal $x_1y_1(z_2e_2+z_3e_3)$, and $\g$ is not a Lie algebra. The map $R(e_1)=e_1$, $R(e_2)=0$, $R(e_3)=e_2$ is a Rota--Baxter operator on $\g$, since $[Rx,Ry]=x_1y_3e_2=R\big(x_1(y_2+y_3)e_2+x_1y_3e_3\big)=R([Rx,y]+[x,Ry])$. Hence $x\lt y=[x,Ry]$ and $x\rt y=[Rx,y]$ define a pre-Leibniz algebra $\aa$ with nonzero products $e_1\lt e_3=e_2$, $e_1\rt e_2=e_2$ and $e_1\rt e_3=e_3$. For any pre-Leibniz algebra, $(\aa,L_\lt,L_\rt,0,0)$ is a representation, since (R1), (R4), (R7) are \eqref{eq:PL1}--\eqref{eq:PL3} and the other identities read $0=0$; moreover $\Id_\aa$ is a Kupershmidt operator with respect to it, and $V_{\Id_\aa}=\aa$. For our $\aa$, $V=(\aa,L_\lt,L_\rt,0,0)$ has the actions $L_\lt(x)=x_1E$ and $L_\rt(x)=x_1P$, where $E(e_3)=e_2$, $E(e_1)=E(e_2)=0$, and $P$ is the projection onto $\mathrm{span}\{e_2,e_3\}$ along $\kk e_1$.
\end{exam}

\section{Deformations of a pre-Leibniz algebra with a representation and Nijenhuis pairs}\label{sec:NP}

In this section, we study infinitesimal deformations of a pre-Leibniz algebra with a representation, which lead to Nijenhuis pairs and dual-Nijenhuis pairs.

Let $(V,\lL,\lR,\rL,\rR)$ be a representation of a pre-Leibniz algebra $(\aa,\lt,\rt)$. Throughout, $\zeta$ runs over the four actions $\{\lL,\lR,\rL,\rR\}$. Let $\omega_\lt,\omega_\rt:\aa\otimes\aa\to\aa$ be bilinear maps and $\varpi_\zeta:\aa\to\gl(V)$ linear maps. For $t\in\kk$, put
\begin{equation}
\begin{aligned}
x\lt_ty&=x\lt y+t\,\omega_\lt(x,y),\qquad x\rt_ty=x\rt y+t\,\omega_\rt(x,y),\qquad \zeta_t(x)=\zeta(x)+t\,\varpi_\zeta(x).
\end{aligned}
\end{equation}
The family $(\omega_\lt,\omega_\rt,\varpi_\zeta)$ is said to generate an \emph{infinitesimal deformation} of $\aa$ with the representation $V$ if, for every $t$, $\aa_t=(\aa,\lt_t,\rt_t)$ is a pre-Leibniz algebra and $V_t=(V,(\zeta_t)_\zeta)$ is a representation of $\aa_t$. Of course, $(\aa,V)$ can be viewed as an infinitesimal deformation with $\omega=0$ and $\varpi=0$.

The following notion generalizes the corresponding notions for Lie algebras \cite{HLS} and Leibniz algebras \cite{SJ}.

\begin{definition}\label{def:equiv}
Two infinitesimal deformations $(\aa'_t,V'_t)$ and $(\aa_t,V_t)$ of a pre-Leibniz algebra $\aa$ with a representation $V$ are \emph{equivalent} if there exist $N\in\gl(\aa)$ and $S\in\gl(V)$ such that $(\Id_\aa+tN,\Id_V+tS)$ is a homomorphism from $(\aa'_t,V'_t)$ to $(\aa_t,V_t)$, i.e.\ for all $x,y\in\aa$,
\begin{align}
(\Id_\aa+tN)(x\lt'_ty)&=(\Id_\aa+tN)x\lt_t(\Id_\aa+tN)y,\label{eq:eq1}\\
(\Id_\aa+tN)(x\rt'_ty)&=(\Id_\aa+tN)x\rt_t(\Id_\aa+tN)y,\label{eq:eq1b}\\
(\Id_V+tS)\,\zeta'_t(x)&=\zeta_t\big((\Id_\aa+tN)x\big)(\Id_V+tS)\qquad\text{for each }\zeta.\label{eq:eq2}
\end{align}
An infinitesimal deformation $(\aa_t,V_t)$ is called \emph{trivial} if it is equivalent to $(\aa,V)$, i.e.\ \eqref{eq:eq1}--\eqref{eq:eq2} hold with $(\aa'_t,V'_t)=(\aa_t,V_t)$ and the undeformed structure on the right-hand sides.
\end{definition}

By direct calculation, if $(\aa_t,V_t)$ is trivial, the coefficients of $t$ and $t^2$ in \eqref{eq:eq1}--\eqref{eq:eq1b} give $\omega_\lt=\lt_N$, $\omega_\rt=\rt_N$ and
\begin{equation}\label{eq:Nij}
\begin{aligned}
Nx\lt Ny&=N(x\lt_Ny),\\
Nx\rt Ny&=N(x\rt_Ny),\qquad\forall\, x,y\in\aa,
\end{aligned}
\end{equation}
where
\begin{equation}\label{eq:Ndef}
\begin{aligned}
x\lt_Ny&:=Nx\lt y+x\lt Ny-N(x\lt y),\\
x\rt_Ny&:=Nx\rt y+x\rt Ny-N(x\rt y).
\end{aligned}
\end{equation}
A linear map $N$ satisfying \eqref{eq:Nij} is called a \emph{Nijenhuis operator} on $\aa$. The deformed products $\lt_N,\rt_N$ of \eqref{eq:Ndef} should not be confused with the products $\lt_K,\rt_K$ induced by a Kupershmidt operator $K$ in Section~\ref{sec:pre}; which one is meant will always be clear from the nature of the operator in the subscript. Moreover, the coefficients of $t$ in \eqref{eq:eq2} give $\varpi_\zeta(x)=\zeta(Nx)+\zeta(x)S-S\zeta(x)$, and those of $t^2$ give, for every $\zeta$ and every $x\in\aa$,
\begin{equation}\label{eq:NP}
\zeta(Nx)S=S\zeta(Nx)+S\zeta(x)S-S^2\zeta(x).
\end{equation}

Now we are ready to define Nijenhuis pairs and dual-Nijenhuis pairs.

\begin{definition}\label{def:NP}
Let $N\in\gl(\aa)$ and $S\in\gl(V)$. We say that $(N,S)$ is a \emph{Nijenhuis pair} on $\aa$ with respect to $V$ when $N$ is a Nijenhuis operator on $\aa$ and $S$ satisfies \eqref{eq:NP} for the four actions $\zeta$.
\end{definition}

For the regular representation and $S=N$, \eqref{eq:NP} for $\zeta=L_\lt$ (resp.\ $R_\lt$) is \eqref{eq:Nij} at $(x,y)$ (resp.\ $(y,x)$), and similarly for $\rt$. Hence $(N,N)$ is a Nijenhuis pair on $\aa$ with the regular representation if and only if $N$ is a Nijenhuis operator.

Before studying the representation part, we recall that a Nijenhuis operator deforms the pre-Leibniz structure. If $N$ is a Nijenhuis operator on a pre-Leibniz algebra $\aa$, then $\aa_N:=(\aa,\lt_N,\rt_N)$ is a pre-Leibniz algebra, $(\aa,\lt+t\lt_N,\rt+t\rt_N)$ is a pre-Leibniz algebra for every $t\in\kk$, and $N:\aa_N\to\aa$ is a morphism. Moreover, $\Id_\aa+tN$ is a morphism from $(\aa,\lt+t\lt_N,\rt+t\rt_N)$ to $(\aa,\lt,\rt)$. Indeed, write $\mu$ for either product $\lt$ or $\rt$, and $\mu_N$ for its deformation \eqref{eq:Ndef}. Expanding both sides and comparing the coefficients of $t$ and $t^2$ by \eqref{eq:Ndef} and \eqref{eq:Nij}, we get, for all $x,y\in\aa$ and $t\in\kk$,
\begin{equation}\label{eq:IdtN}
(\Id_\aa+tN)\big(\mu+t\mu_N\big)(x,y)=\mu\big((\Id_\aa+tN)x,(\Id_\aa+tN)y\big).
\end{equation}
Since $\det(\Id_\aa+tN)$ is a polynomial in $t$ with constant term $1$, $\Id_\aa+tN$ is invertible for all but finitely many $t$. For such $t$, \eqref{eq:IdtN} shows that $\mu+t\mu_N$ is obtained from $\mu$ by transport of structure along $\Id_\aa+tN$, hence satisfies \eqref{eq:PL1}--\eqref{eq:PL3}. These identities are polynomial in $t$ and hold for infinitely many $t$, so they hold for all $t$; their coefficient of $t^2$ shows that $\aa_N$ is a pre-Leibniz algebra. Finally, \eqref{eq:Nij} says that $N(\mu_N(x,y))=\mu(Nx,Ny)$, i.e.\ $N:\aa_N\to\aa$ is a morphism, and \eqref{eq:IdtN} says that $\Id_\aa+tN$ is a morphism.

Nijenhuis pairs can be characterized in terms of Nijenhuis operators on the semidirect product.

\begin{lemma}\label{lem:semi}
Let $N\in\gl(\aa)$ and $S\in\gl(V)$, and let $N+S$ denote the map $(x,u)\mapsto(Nx,Su)$ on $\aa\ltimes V$. Then the products of $(\aa\ltimes V)_{N+S}$ are
\begin{align}
(x,u)\lt_{N+S}(y,v)&=\big(x\lt_Ny,\ \varpi_{\lL}(x)v+\varpi_{\rL}(y)u\big),\label{eq:semi}\\
(x,u)\rt_{N+S}(y,v)&=\big(x\rt_Ny,\ \varpi_{\lR}(x)v+\varpi_{\rR}(y)u\big),\label{eq:semi2}
\end{align}
where $\varpi_\zeta(x)=\zeta(Nx)+\zeta(x)S-S\zeta(x)$. Moreover, the Nijenhuis pairs $(N,S)$ are exactly the pairs for which $N+S$ is a Nijenhuis operator on $\aa\ltimes V$.
\end{lemma}
\begin{proof}
Put $M=N+S$. By the definition of $\aa\ltimes V$, we have
\begin{equation}
\begin{aligned}
(x,u)\lt_M(y,v)&=M(x,u)\lt(y,v)+(x,u)\lt M(y,v)-M\big((x,u)\lt(y,v)\big)\\
&=\big(Nx\lt y+x\lt Ny-N(x\lt y),\\
&\qquad\ [\lL(Nx)+\lL(x)S-S\lL(x)]v+[\rL(Ny)+\rL(y)S-S\rL(y)]u\big)\\
&=\big(x\lt_Ny,\ \varpi_{\lL}(x)v+\varpi_{\rL}(y)u\big).
\end{aligned}
\end{equation}
Furthermore,
\begin{equation}
\begin{aligned}
M(x,u)\lt M(y,v)-M\big((x,u)\lt_M(y,v)\big)&=\big(Nx\lt Ny-N(x\lt_Ny),\\
&\qquad\ [\lL(Nx)S-S\varpi_{\lL}(x)]v+[\rL(Ny)S-S\varpi_{\rL}(y)]u\big),
\end{aligned}
\end{equation}
which vanishes for all $u,v$ if and only if \eqref{eq:Nij} holds and \eqref{eq:NP} holds for $\zeta=\lL,\rL$. The case of $\rt$ is similar, with $\zeta=\lR,\rR$.
\end{proof}

By the computation before Definition~\ref{def:NP}, a trivial infinitesimal deformation gives rise to a Nijenhuis pair. Conversely, we have the following result.

\begin{thm}\label{thm:trivial}
Let $(N,S)$ be a Nijenhuis pair on a pre-Leibniz algebra $\aa$ with a representation $V$. Then
\begin{equation}
\omega_\lt(x,y)=x\lt_Ny,\qquad \omega_\rt(x,y)=x\rt_Ny,\qquad \varpi_\zeta(x)=\zeta(Nx)+\zeta(x)S-S\zeta(x),
\end{equation}
define a trivial infinitesimal deformation of $\aa$ with the representation $V$. Conversely, every trivial infinitesimal deformation of $\aa$ with the representation $V$ is of this form, for some Nijenhuis pair $(N,S)$.
\end{thm}
\begin{proof}
By the discussion before Lemma~\ref{lem:semi}, $(\aa,\lt_t,\rt_t)$ is a pre-Leibniz algebra for every $t$. By Lemma~\ref{lem:semi}, $M=N+S$ is a Nijenhuis operator on $\aa\ltimes V$, and by \eqref{eq:semi},
\begin{equation}
(x,u)\lt(y,v)+t\,(x,u)\lt_M(y,v)=\big(x\lt_ty,\ (\lL)_t(x)v+(\rL)_t(y)u\big),
\end{equation}
and similarly for $\rt$. Thus, by \eqref{eq:IdtN} applied to $M$, the semidirect product $\aa_t\ltimes V_t$ is a pre-Leibniz algebra, i.e.\ $V_t$ is a representation of $\aa_t$. Furthermore, \eqref{eq:eq1}--\eqref{eq:eq1b} are \eqref{eq:IdtN}, and by \eqref{eq:NP},
\begin{equation}
\begin{aligned}
(\Id_V+tS)\zeta_t(x)&=\zeta(x)+t\big(\varpi_\zeta(x)+S\zeta(x)\big)+t^2S\varpi_\zeta(x)\\
&=\zeta(x)+t\big(\zeta(Nx)+\zeta(x)S\big)+t^2\zeta(Nx)S=\zeta\big((\Id_\aa+tN)x\big)(\Id_V+tS).
\end{aligned}
\end{equation}
Thus the deformation is trivial. The converse follows from the computation before Definition~\ref{def:NP}.
\end{proof}

Replacing $S\zeta(x)S-S^2\zeta(x)$ by $\zeta(x)S^2-S\zeta(x)S$ in \eqref{eq:NP}, we obtain a second notion, whose name is explained by Proposition~\ref{thm:dualpair} below.

\begin{definition}\label{def:dNP}
Let $N\in\gl(\aa)$ and $S\in\gl(V)$. We say that $(N,S)$ is a \emph{dual-Nijenhuis pair} on $\aa$ with respect to $V$ when $N$ is a Nijenhuis operator on $\aa$ and, for the four actions $\zeta$ and all $x\in\aa$,
\begin{equation}\label{eq:dNP}
\zeta(Nx)S=S\zeta(Nx)+\zeta(x)S^2-S\zeta(x)S.
\end{equation}
\end{definition}

There is a close relationship between Nijenhuis pairs and dual-Nijenhuis pairs, through the dual representation of Proposition~\ref{prop:dual}.

\begin{proposition}\label{thm:dualpair}
Let $N\in\gl(\aa)$ and $S\in\gl(V)$. Then $(N,S)$ is a Nijenhuis pair with respect to $V$ exactly when $(N,S^{t})$ is a dual-Nijenhuis pair with respect to the dual representation $V^\vee$.
\end{proposition}
\begin{proof}
Let $\zeta$ be an action of $V$. For all $x\in\aa$, $v\in V$ and $\xi\in V^*$, we have
\begin{equation}
\begin{aligned}
&\pair{\xi}{\zeta(Nx)Sv-S\zeta(Nx)v-S\zeta(x)Sv+S^2\zeta(x)v}\\
&\quad=-\pair{S^{t}\zeta^*(Nx)\xi-\zeta^*(Nx)S^{t}\xi-S^{t}\zeta^*(x)S^{t}\xi+\zeta^*(x)(S^{t})^2\xi}{v}.
\end{aligned}
\end{equation}
Thus \eqref{eq:NP} holds for $\zeta$ if and only if \eqref{eq:dNP} holds for $\zeta^*$ and $S^{t}$. Since \eqref{eq:dNP} is linear in the action, and the actions of $V^\vee$ and the maps $\lL^*,\lR^*,\rL^*,\rR^*$ are invertible linear combinations of each other by \eqref{eq:dual}, the conclusion follows.
\end{proof}

By $(V^\vee)^\vee=V$, it follows in the same way that $(N,S)$ is a dual-Nijenhuis pair with $V$ if and only if $(N,S^{t})$ is a Nijenhuis pair with $V^\vee$.

\begin{definition}\label{def:perfect}
A \emph{perfect Nijenhuis pair} on a pre-Leibniz algebra $\aa$ with a representation $V$ is a Nijenhuis pair $(N,S)$ such that, for each $\zeta$ and every $x\in\aa$,
\begin{equation}\label{eq:perfect}
S^2\zeta(x)+\zeta(x)S^2=2S\zeta(x)S.
\end{equation}
\end{definition}

The right-hand sides of \eqref{eq:NP} and \eqref{eq:dNP} differ by $\zeta(x)S^2+S^2\zeta(x)-2S\zeta(x)S$. Hence a perfect Nijenhuis pair is also a dual-Nijenhuis pair, and a pair which is both a Nijenhuis pair and a dual-Nijenhuis pair is perfect.

\begin{remark}\label{prop:semiNij}
By Lemma~\ref{lem:semi}, $N+S$ is a Nijenhuis operator on $\aa\ltimes V$ for every Nijenhuis pair $(N,S)$. If $(N,S)$ is perfect, then $(N,S^{t})$ is a Nijenhuis pair with respect to $V^\vee$ by the remark after Proposition~\ref{thm:dualpair}, so $N+S^{t}$ is a Nijenhuis operator on $\aa\ltimes V^\vee$.
\end{remark}

Let us consider, for each action $\zeta$, the linear maps $\hat\zeta,\tilde\zeta:\aa\to\gl(V)$ defined by
\begin{align}
\hat\zeta(x)&=\zeta(Nx)+\zeta(x)S-S\zeta(x),\label{eq:hattilde}\\
\tilde\zeta(x)&=\zeta(Nx)+S\zeta(x)-\zeta(x)S .\label{eq:hattilde2}
\end{align}

As a consequence of Lemma~\ref{lem:semi} and Proposition~\ref{thm:dualpair}, (dual-)Nijenhuis pairs give rise to new representations of the deformed pre-Leibniz algebra $\aa_N$.

\begin{corollary}\label{cor:reps}
\begin{enumerate}[(i)]
\item Let $(N,S)$ be a Nijenhuis pair on a pre-Leibniz algebra $\aa$ with a representation $V$. Then $\hat V:=(V,\hat\ell_\lt,\hat\ell_\rt,\hat r_\lt,\hat r_\rt)$ is a representation of the pre-Leibniz algebra $\aa_N$.
\item Let $(N,S)$ be a dual-Nijenhuis pair on a pre-Leibniz algebra $\aa$ with a representation $V$. Then $\tilde V:=(V,\tilde\ell_\lt,\tilde\ell_\rt,\tilde r_\lt,\tilde r_\rt)$ is a representation of the pre-Leibniz algebra $\aa_N$.
\end{enumerate}
\end{corollary}
\begin{proof}
(i) By Lemma~\ref{lem:semi}, $N+S$ is a Nijenhuis operator on $\aa\ltimes V$, and by \eqref{eq:semi}--\eqref{eq:semi2} the deformed products are given by
\begin{equation}
(x,u)\lt_{N+S}(y,v)=\big(x\lt_Ny,\ \hat\ell_\lt(x)v+\hat r_\lt(y)u\big),\qquad (x,u)\rt_{N+S}(y,v)=\big(x\rt_Ny,\ \hat\ell_\rt(x)v+\hat r_\rt(y)u\big),
\end{equation}
which implies that $\hat V$ is a representation of $\aa_N$.

(ii) For an action $\xi$ of $V^\vee$, we have $\hat\xi(x)=\xi(Nx)+\xi(x)S^{t}-S^{t}\xi(x)$, and thus
\begin{equation}
(\hat\xi)^*(x)=-\big(\xi(Nx)+\xi(x)S^{t}-S^{t}\xi(x)\big)^{t}=\xi^*(Nx)+S\xi^*(x)-\xi^*(x)S=\widetilde{\xi^*}(x).
\end{equation} Since $(N,S)$ is a dual-Nijenhuis pair, by Proposition~\ref{thm:dualpair} and $(V^\vee)^\vee=V$, $(N,S^{t})$ is a Nijenhuis pair with respect to $V^\vee$. By (i), $\widehat{V^\vee}$ is a representation of $\aa_N$, and thus, by Proposition~\ref{prop:dual}, so is $(\widehat{V^\vee})^\vee=\tilde V$.
\end{proof}

We now give an example of a perfect Nijenhuis pair, and of the corresponding trivial deformation.

\begin{exam}\label{ex:NP3}
Let $\aa$ and $V=(\aa,L_\lt,L_\rt,0,0)$ be as in Example~\ref{ex:3dim}. The linear maps commuting with $E$ and $P$ (that is, with all the actions of $V$) are exactly
\begin{equation}
N(e_1)=\lambda e_1,\qquad N(e_2)=\nu e_2,\qquad N(e_3)=\kappa e_2+\nu e_3\qquad(\lambda,\nu,\kappa\in\kk),
\end{equation}
since $N$ must preserve $\ker P=\kk e_1$ and $\operatorname{im}P$, and the commutant of the nonzero nilpotent map $E|_{\operatorname{im}P}$ is $\{\nu\Id+\kappa E\}$. For such $N$, $x\lt Ny=N(x\lt y)$, hence $x\lt_Ny=Nx\lt y=\lambda\,x\lt y$ and $N(x\lt_Ny)=L_\lt(Nx)Ny=Nx\lt Ny$; similarly for $\rt$. Thus $N$ is a Nijenhuis operator and $\aa_N=(\aa,\lambda\lt,\lambda\rt)$. Since $S=N$ commutes with every $\zeta(x)$, both sides of \eqref{eq:NP}, \eqref{eq:dNP} and \eqref{eq:perfect} reduce to the same expression, so $(N,N)$ is a perfect Nijenhuis pair, and $\hat V=\tilde V=(\aa,\lambda L_\lt,\lambda L_\rt,0,0)$. By Theorem~\ref{thm:trivial}, the deformation $\lt_t=(1+t\lambda)\lt$, $\rt_t=(1+t\lambda)\rt$, $\zeta_t=(1+t\lambda)\zeta$ is trivial; note that $N\notin\kk\Id$ when $\lambda\neq\nu$ or $\kappa\neq0$.
\end{exam}

The following example shows that Nijenhuis pairs and dual-Nijenhuis pairs are different notions.

\begin{exam}\label{ex:NPnotdual}
Let $\aa$ and $N$ be as in Example~\ref{ex:NP3}. Then $(N,N)$ is a Nijenhuis pair with the regular representation. The map $F:=R_\rt(e_2)$ is given by $F(e_1)=e_2$, $F(e_2)=F(e_3)=0$, so $NF-FN=(\nu-\lambda)F$ and $N^2F+FN^2-2NFN=[N,[N,F]]=(\nu-\lambda)^2F$. Since the right-hand sides of \eqref{eq:NP} and \eqref{eq:dNP} differ by $\zeta(x)S^2+S^2\zeta(x)-2S\zeta(x)S$, \eqref{eq:dNP} fails for $\zeta=R_\rt$, $x=e_2$ when $\lambda\neq\nu$. Hence $(N,N)$ is a Nijenhuis pair but not a dual-Nijenhuis pair. Dually, by Proposition~\ref{thm:dualpair} and the remark following it, $(N,N^{t})$ is a dual-Nijenhuis pair with respect to $\aa^\vee$ which is not a Nijenhuis pair when $\lambda\neq\nu$.
\end{exam}

\section{Kupershmidt-(dual-)Nijenhuis structures}\label{sec:ON}

In this section, we introduce Kupershmidt-(dual-)Nijenhuis structures on a pre-Leibniz algebra with a representation.

Let $K$ be a Kupershmidt operator on a pre-Leibniz algebra $(\aa,\lt,\rt)$ with respect to a representation $(V,\lL,\lR,\rL,\rR)$. On the pre-Leibniz algebra $V_K$, let $S\in\gl(V)$ and define the deformed operations
\begin{equation}
u\,(\lt_K)_S\,v=Su\lt_Kv+u\lt_KSv-S(u\lt_Kv),\qquad u\,(\rt_K)_S\,v=Su\rt_Kv+u\rt_KSv-S(u\rt_Kv).
\end{equation}
Now let $(N,S)$ be a Nijenhuis pair or a dual-Nijenhuis pair on $\aa$ with the representation $V$. By Corollary~\ref{cor:reps}, there are two representations $\hat V$ and $\tilde V$ of the deformed pre-Leibniz algebra $\aa_N$. We define operations on $V$ by
\begin{align}
u\lt^K_{\hat\rho}v&=\hat\ell_\lt(Ku)v+\hat r_\lt(Kv)u, & u\rt^K_{\hat\rho}v&=\hat\ell_\rt(Ku)v+\hat r_\rt(Kv)u,\label{eq:hatops}\\
u\lt^K_{\tilde\rho}v&=\tilde\ell_\lt(Ku)v+\tilde r_\lt(Kv)u, & u\rt^K_{\tilde\rho}v&=\tilde\ell_\rt(Ku)v+\tilde r_\rt(Kv)u.\label{eq:hatops3}
\end{align}
In general, these operations are not pre-Leibniz products on $V$; in the following, we study when they are.

\begin{definition}\label{def:ON}
Let $K:V\to\aa$ be a Kupershmidt operator on a pre-Leibniz algebra $\aa$ with respect to a representation $V$, and let $(N,S)$ be a Nijenhuis pair (resp.\ a dual-Nijenhuis pair). The triple $(K,S,N)$ is called a \emph{Kupershmidt-Nijenhuis structure} (resp.\ a \emph{Kupershmidt-dual-Nijenhuis structure}) on $\aa$ with respect to $V$ if
\begin{align}
NK&=KS,\label{eq:ON}\\
u\lt_{NK}v&=u\,(\lt_K)_S\,v,\label{eq:ON2a}\\
u\rt_{NK}v&=u\,(\rt_K)_S\,v,\qquad\forall\,u,v\in V,\label{eq:ON2}
\end{align}
where $u\lt_{NK}v=\lL(NKu)v+\rL(NKv)u$ and $u\rt_{NK}v=\lR(NKu)v+\rR(NKv)u$.
\end{definition}

\begin{remark}\label{rem:scaling}
If $(K,S,N)$ is a Kupershmidt-(dual-)Nijenhuis structure, then so is $(K,cS,cN)$ for every $c\in\kk$, since \eqref{eq:Nij}, \eqref{eq:NP} and \eqref{eq:dNP} are homogeneous in $(N,S)$ and \eqref{eq:ON}--\eqref{eq:ON2} are linear in $(N,S)$.
\end{remark}

As in the Leibniz case \cite{SJ}, we have the following.

\begin{lemma}\label{prop:42}
\begin{enumerate}[(i)]
\item If $(K,S,N)$ is a Kupershmidt-Nijenhuis structure on a pre-Leibniz algebra $\aa$ with respect to a representation $V$, then
$u(\lt_K)_Sv=u\lt^K_{\hat\rho}v$ and $u(\rt_K)_Sv=u\rt^K_{\hat\rho}v$ for all $u,v\in V$.
\item If $(K,S,N)$ is a Kupershmidt-dual-Nijenhuis structure on a pre-Leibniz algebra $\aa$ with respect to a representation $V$, then
$u(\lt_K)_Sv=u\lt^K_{\tilde\rho}v$ and $u(\rt_K)_Sv=u\rt^K_{\tilde\rho}v$ for all $u,v\in V$.
\end{enumerate}
\end{lemma}
\begin{proof}
Since $NK=KS$, we have $Su\lt_Kv=\lL(NKu)v+\rL(Kv)Su$ and $u\lt_KSv=\lL(Ku)Sv+\rL(NKv)u$. Thus
\begin{align}
u\,(\lt_K)_S\,v-u\lt_{NK}v&=Su\lt_Kv+u\lt_KSv-S(u\lt_Kv)-\lL(NKu)v-\rL(NKv)u\notag\\
&=\lL(Ku)Sv+\rL(Kv)Su-S(u\lt_Kv).\label{eq:key1}
\end{align}
Moreover, by \eqref{eq:hattilde}--\eqref{eq:hattilde2}, we have
\begin{align}
u\lt^K_{\hat\rho}v&=\lL(NKu)v+\lL(Ku)Sv-S\lL(Ku)v+\rL(NKv)u+\rL(Kv)Su-S\rL(Kv)u\notag\\
&=u\,(\lt_K)_S\,v,\label{eq:key2}\\
u\lt^K_{\tilde\rho}v&=\lL(NKu)v-\lL(Ku)Sv+S\lL(Ku)v+\rL(NKv)u-\rL(Kv)Su+S\rL(Kv)u\notag\\
&=u\lt_{NK}v+S(u\lt_Kv)-\lL(Ku)Sv-\rL(Kv)Su,\label{eq:key3}
\end{align}
and similarly for $\rt$. (i) follows from \eqref{eq:key2} directly. (ii) By \eqref{eq:key1} and \eqref{eq:key3}, we have $u\,(\lt_K)_S\,v+u\lt^K_{\tilde\rho}v=2\,u\lt_{NK}v$. Then by \eqref{eq:ON2a}--\eqref{eq:ON2}, we obtain $u\,(\lt_K)_S\,v=u\lt^K_{\tilde\rho}v$. The case of $\rt$ is similar.
\end{proof}

In particular, by \eqref{eq:key1}, if $NK=KS$, then \eqref{eq:ON2a}--\eqref{eq:ON2} is equivalent to the identities, for all $u,v\in V$,
\begin{equation}\label{eq:F}
\begin{aligned}
S(u\lt_Kv)&=\lL(Ku)Sv+\rL(Kv)Su,\\
S(u\rt_Kv)&=\lR(Ku)Sv+\rR(Kv)Su.
\end{aligned}
\end{equation}

\begin{proposition}\label{thm:SNij}
If $(K,S,N)$ is a Kupershmidt-Nijenhuis structure or a Kupershmidt-dual-Nijenhuis structure on a pre-Leibniz algebra $\aa$ with respect to a representation $V$, then $S$ is a Nijenhuis operator on the pre-Leibniz algebra $V_K$ and $(V_K)_S=V_{NK}$. Moreover, the operations $\big((\lt_K)_S,(\rt_K)_S\big)$, $(\lt^K_{\hat\rho},\rt^K_{\hat\rho})$ and $(\lt^K_{\tilde\rho},\rt^K_{\tilde\rho})$ all give rise to pre-Leibniz algebras.
\end{proposition}
\begin{proof}
We only prove the conclusion for $\lt$; the case of $\rt$ is similar. For the Kupershmidt-Nijenhuis structure, by \eqref{eq:NP} and substituting $x$ by $Ku$ for $\zeta=\lL$ and by $Kv$ for $\zeta=\rL$, and using $KS=NK$, we get
\begin{equation}
\begin{aligned}
0={}&\lL(NKu)Sv-S\big(\lL(NKu)v+\lL(Ku)Sv-S\lL(Ku)v\big)\\
&+\rL(NKv)Su-S\big(\rL(NKv)u+\rL(Kv)Su-S\rL(Kv)u\big)\\
={}&\lL(KSu)Sv+\rL(KSv)Su\\
&-S\big(\lL(KSu)v+\rL(Kv)Su+\lL(Ku)Sv+\rL(KSv)u-S(\lL(Ku)v+\rL(Kv)u)\big)\\
={}&Su\lt_KSv-S\big(Su\lt_Kv+u\lt_KSv-S(u\lt_Kv)\big),
\end{aligned}
\end{equation}
which implies that $S$ is a Nijenhuis operator on $V_K$.

For the Kupershmidt-dual-Nijenhuis structure, the proof is not direct. In fact, by \eqref{eq:key1} and \eqref{eq:ON2a}--\eqref{eq:ON2}, the identity \eqref{eq:F} holds, which gives, at $(Su,v)$, at $(u,Sv)$ and after applying $S$,
\begin{align}
S\big(\lL(KSu)v\big)+S\big(\rL(Kv)Su\big)&=\lL(KSu)Sv+\rL(Kv)S^2u,\label{eq:dn1}\\
S\big(\lL(Ku)Sv\big)+S\big(\rL(KSv)u\big)&=\lL(Ku)S^2v+\rL(KSv)Su,\label{eq:dn2}\\
S^2\big(\lL(Ku)v\big)+S^2\big(\rL(Kv)u\big)&=S\big(\lL(Ku)Sv\big)+S\big(\rL(Kv)Su\big).\label{eq:dn3}
\end{align}
By \eqref{eq:dNP} for $\zeta=\lL$ at $x=Ku$ and for $\zeta=\rL$ at $x=Kv$, we have
\begin{equation}\label{eq:dn4}
\begin{aligned}
&\lL(NKu)Sv-S\big(\lL(NKu)v\big)+\rL(NKv)Su-S\big(\rL(NKv)u\big)\\
&\quad=\lL(Ku)S^2v-S\big(\lL(Ku)Sv\big)+\rL(Kv)S^2u-S\big(\rL(Kv)Su\big).
\end{aligned}
\end{equation}
By \eqref{eq:dn1}--\eqref{eq:dn4}, \eqref{eq:F} and $KS=NK$, we have
\begin{align}
&Su\lt_KSv-S\big(u\,(\lt_K)_S\,v\big)\notag\\\displaybreak[0]
&=\lL(KSu)Sv+\rL(KSv)Su-S\big(\lL(KSu)v+\rL(Kv)Su+\lL(Ku)Sv+\rL(KSv)u\big)\notag\\\displaybreak[0]
&\quad+S^2\big(\lL(Ku)v+\rL(Kv)u\big)\notag\\\displaybreak[0]
&=\lL(KSu)Sv+\rL(KSv)Su-S\big(\lL(KSu)v\big)-S\big(\rL(KSv)u\big)\notag\\\displaybreak[0]
&=\lL(Ku)S^2v-S\big(\lL(Ku)Sv\big)+\rL(Kv)S^2u-S\big(\rL(Kv)Su\big)\notag\\\displaybreak[0]
&=S\big(\lL(Ku)Sv\big)+S\big(\rL(KSv)u\big)-\rL(KSv)Su-S\big(\lL(Ku)Sv\big)\notag\\\displaybreak[0]
&\quad+S\big(\lL(KSu)v\big)+S\big(\rL(Kv)Su\big)-\lL(KSu)Sv-S\big(\rL(Kv)Su\big)\notag\\\displaybreak[0]
&=S\big(\lL(KSu)v+\rL(KSv)u\big)-\big(\lL(KSu)Sv+\rL(KSv)Su\big)\notag\\\displaybreak[0]
&=S\big(Su\lt_Kv+u\lt_KSv\big)-S\big(\lL(Ku)Sv+\rL(Kv)Su\big)-Su\lt_KSv\notag\\\displaybreak[0]
&=S\big(u\,(\lt_K)_S\,v\big)-Su\lt_KSv .
\end{align}
Since $\operatorname{char}\kk=0$, we get $Su\lt_KSv=S\big(u\,(\lt_K)_S\,v\big)$. Thus $S$ is a Nijenhuis operator on $V_K$.

In both cases, by \eqref{eq:ON2a}--\eqref{eq:ON2}, $(V_K)_S=V_{NK}$. Moreover, by \eqref{eq:key1}--\eqref{eq:key3} and \eqref{eq:ON2a}--\eqref{eq:ON2}, the three operations in the statement coincide with the products of $V_{NK}$, so they are all pre-Leibniz products.
\end{proof}

A Kupershmidt-(dual-)Nijenhuis structure also produces new Kupershmidt operators, as the following result shows.

\begin{thm}\label{prop:ONprop}
Let $V$ be a representation of a pre-Leibniz algebra $\aa$, and suppose that $(K,S,N)$ is a Kupershmidt-Nijenhuis structure (resp.\ a Kupershmidt-dual-Nijenhuis structure) on $\aa$ with respect to $V$. Then
\begin{enumerate}[(i)]
\item $K$ is a Kupershmidt operator on the pre-Leibniz algebra $\aa_N$ with respect to the representation $\hat V$ (resp.\ $\tilde V$);
\item $NK$ is a Kupershmidt operator on $\aa$ with respect to the representation $V$.
\end{enumerate}
\end{thm}
\begin{proof}
We only prove the theorem for the Kupershmidt-Nijenhuis structure. The other one can be proved similarly.

(i) Since $K$ is a Kupershmidt operator and $KS=NK$, by Lemma~\ref{prop:42}, we have
\begin{equation}
\begin{aligned}
K\big(u\lt^K_{\hat\rho}v\big)&=K\big(u\,(\lt_K)_S\,v\big)=K\big(Su\lt_Kv+u\lt_KSv-S(u\lt_Kv)\big)\\
&=KSu\lt Kv+Ku\lt KSv-KS(u\lt_Kv)\\
&=NKu\lt Kv+Ku\lt NKv-N(Ku\lt Kv)=Ku\lt_NKv,
\end{aligned}
\end{equation}
and similarly for $\rt$. Thus $K$ is a Kupershmidt operator on $\aa_N$ with respect to $\hat V$.

(ii) By \eqref{eq:ON2a}--\eqref{eq:ON2} and (i), we have
\begin{equation}
NK(u\lt_{NK}v)=NK\big(u\,(\lt_K)_S\,v\big)=N(Ku\lt_NKv)=NKu\lt NKv,
\end{equation}
and similarly for $\rt$, which implies that $NK$ is a Kupershmidt operator on $\aa$ with respect to $V$.
\end{proof}

For an invertible (in fact, surjective) Kupershmidt operator, the two notions of Kupershmidt-Nijenhuis and Kupershmidt-dual-Nijenhuis structures are related as follows.

\begin{thm}\label{thm:invON}
Suppose that $(K,S,N)$ is a Kupershmidt-Nijenhuis structure on a pre-Leibniz algebra $\aa$ with respect to a representation $V$, with $K$ surjective (for instance, invertible). Then $(K,S,N)$ is also a Kupershmidt-dual-Nijenhuis structure. Consequently, $(N,S)$ is a perfect Nijenhuis pair.
\end{thm}
\begin{proof}
We only need to prove that the Nijenhuis pair $(N,S)$ is also a dual-Nijenhuis pair. By \eqref{eq:key1} and \eqref{eq:ON2a}--\eqref{eq:ON2}, the identity \eqref{eq:F} holds. Since $S$ is a Nijenhuis operator on $V_K$ and $u\,(\lt_K)_S\,v=u\lt_{NK}v$, by Proposition~\ref{thm:SNij}, we have $S(u\lt_{NK}v)=Su\lt_KSv$, which means that
\begin{equation}\label{eq:G}
S\big(\lL(NKu)v\big)+S\big(\rL(NKv)u\big)=\lL(NKu)Sv+\rL(NKv)Su .
\end{equation}
By \eqref{eq:F} at $(Su,v)$ and $KS=NK$, we have
\begin{equation}
S\big(\lL(NKu)v\big)-\lL(NKu)Sv=\rL(Kv)S^2u-S\big(\rL(Kv)Su\big).
\end{equation}
Thus \eqref{eq:G} implies that
\begin{equation}
\begin{aligned}
0&=S\big(\lL(NKu)v\big)-\lL(NKu)Sv+S\big(\rL(NKv)u\big)-\rL(NKv)Su\\
&=\rL(Kv)S^2u-S\big(\rL(Kv)Su\big)+S\big(\rL(NKv)u\big)-\rL(NKv)Su .
\end{aligned}
\end{equation}
Since $K$ is surjective, letting $x=Kv$, we get $\rL(Nx)Su=S\big(\rL(Nx)u\big)+\rL(x)S^2u-S\big(\rL(x)Su\big)$, which is \eqref{eq:dNP} for $\zeta=\rL$. Similarly, by \eqref{eq:G} and \eqref{eq:F} at $(u,Sv)$, we obtain \eqref{eq:dNP} for $\zeta=\lL$, and the $\rt$-parts give \eqref{eq:dNP} for $\zeta=\lR,\rR$. Thus $(N,S)$ is a dual-Nijenhuis pair, and hence a perfect Nijenhuis pair. We finish the proof.
\end{proof}

\section{Compatible Kupershmidt operators and hierarchies}\label{sec:hier}

In this section, we relate Kupershmidt-(dual-)Nijenhuis structures to compatible Kupershmidt operators, and construct hierarchies of pairwise compatible Kupershmidt operators.

\subsection{Compatibility of Kupershmidt operators}\label{subsec:compatN}

We now turn to compatible Kupershmidt operators, which are closely related to Kupershmidt-(dual-)Nijenhuis structures.

\begin{definition}\label{def:compat}
Let $K_1,K_2$ be two Kupershmidt operators on a pre-Leibniz algebra $\aa$ with respect to a representation $V$. We say that $K_1$ and $K_2$ are \emph{compatible} if $n_1K_1+n_2K_2$ is a Kupershmidt operator for all $n_1,n_2\in\kk$.
\end{definition}

\begin{proposition}\label{prop:compateq}
Let $K_1,K_2$ be two Kupershmidt operators on a pre-Leibniz algebra $\aa$ with respect to a representation $V$. Then $K_1$ and $K_2$ are compatible if and only if, for all $u,v\in V$,
\begin{align}
K_1u\lt K_2v+K_2u\lt K_1v&=K_1(u\lt_{K_2}v)+K_2(u\lt_{K_1}v),\label{eq:compat}\\
K_1u\rt K_2v+K_2u\rt K_1v&=K_1(u\rt_{K_2}v)+K_2(u\rt_{K_1}v).\label{eq:compat2}
\end{align}
\end{proposition}
\begin{proof}
Let $n_1,n_2\in\kk$ and $K=n_1K_1+n_2K_2$. Since $u\lt_Kv=n_1\,u\lt_{K_1}v+n_2\,u\lt_{K_2}v$, we have
\begin{equation}
\begin{aligned}
Ku\lt Kv-K(u\lt_Kv)={}&n_1^2\big(K_1u\lt K_1v-K_1(u\lt_{K_1}v)\big)+n_2^2\big(K_2u\lt K_2v-K_2(u\lt_{K_2}v)\big)\\
&+n_1n_2\big(K_1u\lt K_2v+K_2u\lt K_1v-K_1(u\lt_{K_2}v)-K_2(u\lt_{K_1}v)\big)\\
={}&n_1n_2\big(K_1u\lt K_2v+K_2u\lt K_1v-K_1(u\lt_{K_2}v)-K_2(u\lt_{K_1}v)\big),
\end{aligned}
\end{equation}
since $K_1$ and $K_2$ are Kupershmidt operators. Thus $Ku\lt Kv=K(u\lt_Kv)$ for all $n_1,n_2$ if and only if \eqref{eq:compat} holds. Similarly for $\rt$ with \eqref{eq:compat2}.
\end{proof}

Composing a Kupershmidt operator with a Nijenhuis operator is a natural way to obtain a second Kupershmidt operator; the next result, a pre-Leibniz version of \cite[Proposition~4.3]{HLS}, describes when this works and when the two operators are compatible.

\begin{proposition}\label{prop:NKcompat}
Let $K:V\to\aa$ be a Kupershmidt operator on a pre-Leibniz algebra $\aa$ with respect to a representation $V$, and let $N$ be a Nijenhuis operator on $\aa$. The composition $NK$ is again a Kupershmidt operator with respect to $V$ precisely when, for all $u,v\in V$,
\begin{align}
N\big(NKu\lt Kv+Ku\lt NKv\big)&=N\big(K(u\lt_{NK}v)+NK(u\lt_Kv)\big),\label{eq:NKcomp1}\\
N\big(NKu\rt Kv+Ku\rt NKv\big)&=N\big(K(u\rt_{NK}v)+NK(u\rt_Kv)\big).\label{eq:NKcomp2}
\end{align}
In this case, if moreover $N$ is invertible, then $K$ and $NK$ are compatible.
\end{proposition}
\begin{proof}
Since $N$ is a Nijenhuis operator, we have
\begin{equation}
NKu\lt NKv=N\big(NKu\lt Kv+Ku\lt NKv\big)-N^2(Ku\lt Kv).
\end{equation}
Note that $Ku\lt Kv=K(u\lt_Kv)$. Then $NKu\lt NKv=NK(u\lt_{NK}v)$ if and only if \eqref{eq:NKcomp1} holds, and similarly for $\rt$. If $NK$ is a Kupershmidt operator and $N$ is invertible, then we have
\begin{equation}
NKu\lt Kv+Ku\lt NKv=K(u\lt_{NK}v)+NK(u\lt_Kv),
\end{equation}
and similarly for $\rt$, which is exactly the condition that $K$ and $NK$ are compatible.
\end{proof}

In the opposite direction, when one of two compatible Kupershmidt operators is invertible, their quotient is a Nijenhuis operator.

\begin{proposition}\label{prop:compatNij}
Let $K_1,K_2$ be two compatible Kupershmidt operators on a pre-Leibniz algebra $\aa$ with respect to a representation $V$, and suppose that $K_2$ is invertible. Then $N=K_1K_2^{-1}$ is a Nijenhuis operator on $\aa$.
\end{proposition}
\begin{proof}
For all $x,y\in\aa$, there exist $u,v\in V$ such that $K_2u=x$ and $K_2v=y$. Hence $N=K_1K_2^{-1}$ is a Nijenhuis operator if and only if
\begin{equation}
NK_2u\lt NK_2v=N\big(NK_2u\lt K_2v+K_2u\lt NK_2v\big)-N^2(K_2u\lt K_2v),
\end{equation}
and similarly for $\rt$. Since $K_1=NK_2$ is a Kupershmidt operator, the left-hand side is $NK_2(u\lt_{NK_2}v)$. Since $K_2$ is a Kupershmidt operator which is compatible with $K_1=NK_2$, we have
\begin{equation}
\begin{aligned}
NK_2u\lt K_2v+K_2u\lt NK_2v&=K_2(u\lt_{NK_2}v)+NK_2(u\lt_{K_2}v)\\
&=K_2(u\lt_{NK_2}v)+N(K_2u\lt K_2v).
\end{aligned}
\end{equation}
Applying $N$ to both sides, we get the conclusion.
\end{proof}

Combining Propositions~\ref{prop:NKcompat} and~\ref{prop:compatNij}, we obtain the following characterization.

\begin{corollary}\label{cor:invcompat}
Let $K_1,K_2$ be two invertible Kupershmidt operators on a pre-Leibniz algebra $\aa$ with respect to a representation $V$. Then $K_1$ and $K_2$ are compatible if and only if $N=K_1K_2^{-1}$ is a Nijenhuis operator on $\aa$.
\end{corollary}
\begin{proof}
It follows from Propositions~\ref{prop:NKcompat} and~\ref{prop:compatNij}.
\end{proof}

\begin{corollary}\label{cor:invRB}
Let $R_1,R_2$ be two invertible Rota--Baxter operators on a pre-Leibniz algebra $\aa$. Then $R_1R_2^{-1}$ is a Nijenhuis operator on $\aa$ exactly when all the linear combinations $n_1R_1+n_2R_2$ ($n_1,n_2\in\kk$) are Rota--Baxter operators.
\end{corollary}

\subsection{Hierarchies generated by Kupershmidt-(dual-)Nijenhuis structures}\label{subsec:hierarchy}

We now return to Kupershmidt-(dual-)Nijenhuis structures $(K,S,N)$. Theorem~\ref{prop:ONprop} already tells us that both $K$ and $NK=KS$ are Kupershmidt operators; our first observation is that this pair is compatible.

\begin{proposition}\label{prop:47}
Suppose that $K$ is a Kupershmidt operator on a pre-Leibniz algebra $\aa$ with respect to a representation $V$. If $(K,S,N)$ is a Kupershmidt-Nijenhuis structure or a Kupershmidt-dual-Nijenhuis structure, then $K$ and $KS=NK$ are compatible Kupershmidt operators.
\end{proposition}
\begin{proof}
We only prove the conclusion for $\lt$; the case of $\rt$ is similar. It is sufficient to prove that $K+KS$ is a Kupershmidt operator. It is obvious that
\begin{equation}
u\lt_{K+KS}v=u\lt_Kv+u\lt_{KS}v=u\lt_Kv+u\,(\lt_K)_S\,v .
\end{equation}
Since $K$ and $KS=NK$ are Kupershmidt operators by Theorem~\ref{prop:ONprop}(ii), by \eqref{eq:ON2a}--\eqref{eq:ON2}, we have
\begin{align}
&(K+KS)\big(u\lt_{K+KS}v\big)\notag\\\displaybreak[0]
&=K(u\lt_Kv)+KS\big(u\,(\lt_K)_S\,v\big)+KS(u\lt_Kv)+K\big(u\,(\lt_K)_S\,v\big)\notag\\\displaybreak[0]
&=K(u\lt_Kv)+KS\big(u\,(\lt_K)_S\,v\big)+KS(u\lt_Kv)+K\big(Su\lt_Kv+u\lt_KSv-S(u\lt_Kv)\big)\notag\\\displaybreak[0]
&=K(u\lt_Kv)+KS\big(u\,(\lt_K)_S\,v\big)+K(Su\lt_Kv)+K(u\lt_KSv)\notag\\\displaybreak[0]
&=Ku\lt Kv+KSu\lt KSv+KSu\lt Kv+Ku\lt KSv\notag\\\displaybreak[0]
&=(K+KS)u\lt(K+KS)v,
\end{align}
which means that $K+KS$ is a Kupershmidt operator. By Remark~\ref{rem:scaling}, the same argument applied to $(K,cS,cN)$ shows that $n_1K+n_2KS$ is a Kupershmidt operator for all $n_1,n_2\in\kk$. We finish the proof.
\end{proof}

Proposition~\ref{prop:47} is the first step of a hierarchy of pairwise compatible Kupershmidt operators, as in the Lie case \cite[Section~4]{HLS}; our proof avoids powers of Nijenhuis operators.

\begin{lemma}\label{lem:NKstruct}
If $(K,S,N)$ is a Kupershmidt-Nijenhuis structure (resp.\ a Kupershmidt-dual-Nijenhuis structure) on a pre-Leibniz algebra $\aa$ with respect to a representation $V$, then so is $(NK,S,N)$.
\end{lemma}
\begin{proof}
By Theorem~\ref{prop:ONprop}(ii), $NK$ is a Kupershmidt operator, and $N(NK)=NKS=(NK)S$. By \eqref{eq:key1}, it is sufficient to prove \eqref{eq:F} for $NK$. By Proposition~\ref{thm:SNij} and $KS=NK$, we have
\begin{equation}
S(u\lt_{NK}v)=S\big(u\,(\lt_K)_S\,v\big)=Su\lt_KSv=\lL(NKu)Sv+\rL(NKv)Su,
\end{equation}
and similarly for $\rt$. We finish the proof.
\end{proof}

Given a Kupershmidt-(dual-)Nijenhuis structure $(K,S,N)$, we put
\begin{equation}
K_k:=N^kK=KS^k,\qquad k\in\mathbb N,
\end{equation}
where the equality $N^kK=KS^k$ follows from $NK=KS$ by induction, and $K_0=K$. By Lemma~\ref{lem:NKstruct} and induction on $k$, each $(K_k,S,N)$ is a Kupershmidt-Nijenhuis structure (resp.\ a Kupershmidt-dual-Nijenhuis structure). For $i\in\mathbb N$, we denote by $(\lt_{K})_{S^i}$ and $(\rt_K)_{S^i}$ the products on $V$ deformed by $S^i$, defined as in Section~\ref{sec:ON} with $S$ replaced by $S^i$.

\begin{lemma}\label{lem:Si}
Let $(K,S,N)$ be a Kupershmidt-(dual-)Nijenhuis structure on a pre-Leibniz algebra $\aa$ with respect to a representation $V$. Then, for all $i\in\mathbb N$ and $u,v\in V$,
\begin{equation}\label{eq:Si1}
\begin{aligned}
S^i(u\lt_Kv)&=\lL(Ku)S^iv+\rL(Kv)S^iu,\\
S^i(u\rt_Kv)&=\lR(Ku)S^iv+\rR(Kv)S^iu.
\end{aligned}
\end{equation}
\end{lemma}
\begin{proof}
We only prove \eqref{eq:Si1} by induction on $i$. The proof of the second identity is similar. For $i=0$ it is obvious, and for $i=1$, \eqref{eq:Si1} is \eqref{eq:F}. Assume that \eqref{eq:Si1} holds for $i\geq1$. Since $KS^i=K_i$, by \eqref{eq:F} at $(u,S^iv)$ and at $(S^iu,v)$, we have
\begin{equation}
\begin{aligned}
S^{i+1}(u\lt_Kv)&=S\lL(Ku)S^iv+S\rL(Kv)S^iu\\
&=S(u\lt_KS^iv)-S\rL(K_iv)u+S(S^iu\lt_Kv)-S\lL(K_iu)v\\
&=\lL(Ku)S^{i+1}v+\rL(Kv)S^{i+1}u+\lL(K_iu)Sv+\rL(K_iv)Su-S(u\lt_{K_i}v).
\end{aligned}
\end{equation}
By Lemma~\ref{lem:NKstruct}, $(K_i,S,N)$ is a Kupershmidt-(dual-)Nijenhuis structure, so \eqref{eq:F} holds for $K_i$. Thus the last three terms vanish, and \eqref{eq:Si1} holds for $i+1$.
\end{proof}

\begin{proposition}\label{thm:hierarchy}
Let $(K,S,N)$ be a Kupershmidt-Nijenhuis structure (resp.\ a Kupershmidt-dual-Nijenhuis structure) on a pre-Leibniz algebra $\aa$ with respect to a representation $V$, and put $K_k=N^kK=KS^k$. Then:
\begin{enumerate}[(i)]
\item for every $k\in\mathbb N$, $K_k$ is a Kupershmidt operator on $\aa$ with respect to $V$, and $(K_k,S,N)$ is a Kupershmidt-Nijenhuis structure (resp.\ a Kupershmidt-dual-Nijenhuis structure);
\item for all $k,i\in\mathbb N$, the pre-Leibniz algebra $V_{K_{k+i}}$ is the deformation of $V_{K_k}$ by $S^i$, i.e.
\begin{equation}
u\lt_{K_{k+i}}v=u\,(\lt_{K_k})_{S^i}\,v,\qquad u\rt_{K_{k+i}}v=u\,(\rt_{K_k})_{S^i}\,v,
\end{equation}
and $S^i:V_{K_{k+i}}\to V_{K_k}$ is a morphism of pre-Leibniz algebras;
\item for all $k,l\in\mathbb N$, the Kupershmidt operators $K_k$ and $K_l$ are compatible.
\end{enumerate}
\end{proposition}
\begin{proof}
We only prove the conclusion for $\lt$. The case of $\rt$ is similar.

(i) It follows from Lemma~\ref{lem:NKstruct} by induction on $k$.

(ii) Since $K_kS^i=K_{k+i}$, by Lemma~\ref{lem:Si} for the structure $(K_k,S,N)$, we have
\begin{equation}
\begin{aligned}
u\,(\lt_{K_k})_{S^i}\,v&=S^iu\lt_{K_k}v+u\lt_{K_k}S^iv-S^i(u\lt_{K_k}v)\\
&=\lL(K_{k+i}u)v+\rL(K_kv)S^iu+\lL(K_ku)S^iv+\rL(K_{k+i}v)u-S^i(u\lt_{K_k}v)\\
&=\lL(K_{k+i}u)v+\rL(K_{k+i}v)u=u\lt_{K_{k+i}}v .
\end{aligned}
\end{equation}
Moreover, by Lemma~\ref{lem:Si} for the structure $(K_{k+i},S,N)$, we have
\begin{equation}
S^iu\lt_{K_k}S^iv=\lL(K_{k+i}u)S^iv+\rL(K_{k+i}v)S^iu=S^i(u\lt_{K_{k+i}}v),
\end{equation}
which implies that $S^i$ is a morphism from $V_{K_{k+i}}$ to $V_{K_k}$.

(iii) We may assume that $l=k+i$. Since $K_k$ is a Kupershmidt operator and $K_{k+i}=K_kS^i$, by (ii), we have
\begin{equation}
\begin{aligned}
K_ku\lt K_{k+i}v+K_{k+i}u\lt K_kv&=K_k\big(u\lt_{K_k}S^iv+S^iu\lt_{K_k}v\big)\\
&=K_k\big(u\,(\lt_{K_k})_{S^i}\,v\big)+K_kS^i(u\lt_{K_k}v)\\
&=K_k(u\lt_{K_{k+i}}v)+K_{k+i}(u\lt_{K_k}v).
\end{aligned}
\end{equation}
Thus $K_k$ and $K_{k+i}$ are compatible by Proposition~\ref{prop:compateq}. We finish the proof.
\end{proof}

We conclude with a converse: if one of two compatible Kupershmidt operators is invertible, they can be recovered from a Kupershmidt-dual-Nijenhuis structure.

\begin{proposition}\label{prop:48}
Let $K_1,K_2$ be two compatible Kupershmidt operators on a pre-Leibniz algebra $\aa$ with respect to a representation $V$, and suppose that $K_1$ is invertible. Put $S=K_1^{-1}K_2$ and $N=K_2K_1^{-1}$. Then
\begin{enumerate}[(i)]
\item $(K_1,S,N)$ is a Kupershmidt-dual-Nijenhuis structure;
\item $(K_2,S,N)$ is a Kupershmidt-dual-Nijenhuis structure.
\end{enumerate}
\end{proposition}
\begin{proof}
(i) It is obvious that $NK_1=K_2=K_1S$. Thus we only need to prove that $(N,S)$ is a dual-Nijenhuis pair and that \eqref{eq:ON2a}--\eqref{eq:ON2} holds for $K=K_1$. By Proposition~\ref{prop:compatNij}, with the roles of $K_1$ and $K_2$ interchanged, $N=K_2K_1^{-1}$ is a Nijenhuis operator. By Proposition~\ref{prop:compateq}, \eqref{eq:compat} holds. Substituting $K_2$ with $K_1S$ in \eqref{eq:compat}, we have
\begin{equation}\label{eq:pf48}
K_1u\lt K_1Sv+K_1Su\lt K_1v=K_1(u\lt_{K_1S}v)+K_1S(u\lt_{K_1}v).
\end{equation}
Since $K_1$ is a Kupershmidt operator, the left-hand side of \eqref{eq:pf48} is $K_1\big(\lL(K_1u)Sv+\rL(K_1Sv)u+\lL(K_1Su)v+\rL(K_1v)Su\big)$. Since $K_1$ is invertible, \eqref{eq:pf48} is equivalent to
\begin{equation}\label{eq:pf48b}
S\big(\lL(K_1u)v+\rL(K_1v)u\big)=\lL(K_1u)Sv+\rL(K_1v)Su .
\end{equation}
On the other hand, by \eqref{eq:key1}, we have
\begin{equation}
u\,(\lt_{K_1})_S\,v-u\lt_{K_1S}v=\lL(K_1u)Sv+\rL(K_1v)Su-S\big(\lL(K_1u)v+\rL(K_1v)u\big).
\end{equation}
Thus \eqref{eq:pf48b} implies that $u\,(\lt_{K_1})_S\,v=u\lt_{NK_1}v$, and similarly for $\rt$, i.e.\ \eqref{eq:ON2a}--\eqref{eq:ON2} holds. Moreover, since $K_1$ and $K_2=K_1S$ are Kupershmidt operators, $K_1(Su\lt_{K_1}Sv)=K_2u\lt K_2v=K_2(u\lt_{K_2}v)=K_1S(u\lt_{NK_1}v)$, which implies that \eqref{eq:G} holds for $K=K_1$. Since \eqref{eq:F} holds by \eqref{eq:pf48b}, the proof of $(N,S)$ being a dual-Nijenhuis pair is the same as the proof of Theorem~\ref{thm:invON}. Therefore, $(K_1,S,N)$ is a Kupershmidt-dual-Nijenhuis structure.

(ii) Since $NK_2=K_2K_1^{-1}K_2=K_2S$, by direct calculation, we have
\begin{equation}
\begin{aligned}
u\,(\lt_{K_2})_S\,v-u\lt_{K_2S}v
&=\lL(K_2u)Sv+\rL(K_2v)Su-S\big(\lL(K_2u)v+\rL(K_2v)u\big)\\
&=\lL(K_1Su)Sv+\rL(K_1Sv)Su-S\big(\lL(K_1Su)v+\rL(K_1Sv)u\big)\\
&=Su\lt_{K_1}Sv-S(u\lt_{NK_1}v)=0,
\end{aligned}
\end{equation}
and similarly for $\rt$. Thus $(K_2,S,N)$ is also a Kupershmidt-dual-Nijenhuis structure.
\end{proof}

\section{Rota--Baxter--Nijenhuis structures, \texorpdfstring{$r$-$n$}{r-n} structures and BN-structures}\label{sec:rn}

In this section, we show that RBN-structures, $r$-$n$ structures and BN-structures are Kupershmidt-(dual-)Nijenhuis structures, and that RBN-structures and $r$-$n$ structures correspond to each other on quadratic pre-Leibniz algebras.

For a linear map $N:\aa\to\aa$ we denote by $N^{t}:\aa^*\to\aa^*$ its transpose. For a linear map $P:\aa^*\to\aa$, the products $\alpha\lt_P\beta$ and $\alpha\rt_P\beta$ on $\aa^*$ are defined as in \eqref{eq:O}, using the dual regular representation $\aa^\vee$ of \eqref{eq:dualreg}.

\begin{definition}\label{def:rn}
Let $\pi$ be a pre-Leibniz $r$-matrix and $N:\aa\to\aa$ a Nijenhuis operator on a pre-Leibniz algebra $(\aa,\lt,\rt)$. We say that $(\pi,N)$ is a \emph{pre-Leibniz $r$-$n$ structure} on $\aa$ if, for all $\alpha,\beta\in\aa^*$,
\begin{align}
N\pi^\sharp&=\pi^\sharp N^{t},\label{eq:rn1}\\
\alpha\lt_{N\pi^\sharp}\beta&=\alpha\,(\lt_{\pi^\sharp})_{N^{t}}\,\beta,\label{eq:rn2a}\\
\alpha\rt_{N\pi^\sharp}\beta&=\alpha\,(\rt_{\pi^\sharp})_{N^{t}}\,\beta.\label{eq:rn2}
\end{align}
\end{definition}

We immediately have the following.

\begin{thm}\label{thm:rnON}
Suppose that $(\pi,N)$ is a pre-Leibniz $r$-$n$ structure on a pre-Leibniz algebra $\aa$. Then $(\pi^\sharp,N^{t},N)$ is a Kupershmidt-dual-Nijenhuis structure on $\aa$ with respect to the dual regular representation $\aa^\vee$. Consequently, $\pi^\sharp$ and $N\pi^\sharp$ are compatible Kupershmidt operators.
\end{thm}
\begin{proof}
Since $N$ is a Nijenhuis operator, $(N,N)$ is a Nijenhuis pair with respect to the regular representation, and thus $(N,N^{t})$ is a dual-Nijenhuis pair with respect to $\aa^\vee$ by Proposition~\ref{thm:dualpair}. Moreover, \eqref{eq:rn1}--\eqref{eq:rn2} are exactly \eqref{eq:ON}--\eqref{eq:ON2} for $K=\pi^\sharp$ and $S=N^{t}$. The last assertion follows from Proposition~\ref{prop:47}.
\end{proof}

\begin{exam}\label{ex:rnnotKN}
Let $\aa$ and $N$ be as in Examples~\ref{ex:3dim} and~\ref{ex:NP3}, and let $\pi=e_1\otimes e_1$, i.e.\ $\pi^\sharp(e^1)=e_1$ and $\pi^\sharp(e^2)=\pi^\sharp(e^3)=0$. By \eqref{eq:dualreg}, for $\alpha,\beta\in\aa^*$ we have
\begin{equation}
\alpha\lt_{\pi^\sharp}\beta=\alpha_1\beta_2e^3+\beta_1\big(\alpha_2e^2+(\alpha_2+\alpha_3)e^3\big),\qquad \alpha\rt_{\pi^\sharp}\beta=-\alpha_1\big(\beta_2e^2+(\beta_2+\beta_3)e^3\big)-\beta_1\alpha_2e^3 .
\end{equation}
These have no $e^1$-component and $e_1\lt e_1=e_1\rt e_1=0$, so $\pi$ is a pre-Leibniz $r$-matrix. Moreover, $N\pi^\sharp=\pi^\sharp N^{t}=\lambda\pi^\sharp$, and a direct computation shows that \eqref{eq:F} hold for $K=\pi^\sharp$ and $S=N^{t}$. Hence $(\pi,N)$ is a pre-Leibniz $r$-$n$ structure, and $(\pi^\sharp,N^{t},N)$ is a Kupershmidt-dual-Nijenhuis structure with respect to $\aa^\vee$. If $\lambda\neq\nu$, then $(N,N^{t})$ is not a Nijenhuis pair by Example~\ref{ex:NPnotdual}, so this structure is not Kupershmidt-Nijenhuis; its hierarchy $N^k\pi^\sharp=\lambda^k\pi^\sharp$ is trivial.
\end{exam}

Similarly, Rota--Baxter operators combined with Nijenhuis operators lead to the following notion.

\begin{definition}\label{def:RBN}
Let $R:\aa\to\aa$ be a Rota--Baxter operator and $N:\aa\to\aa$ a Nijenhuis operator on a pre-Leibniz algebra $(\aa,\lt,\rt)$. A pair $(R,N)$ is called an \emph{RBN-structure} on $\aa$ if
\begin{align}
NR&=RN,\label{eq:RBN1}\\
x\lt_{NR}y&=x\,(\lt_R)_N\,y,\label{eq:RBN2a}\\
x\rt_{NR}y&=x\,(\rt_R)_N\,y,\qquad\forall\,x,y\in\aa,\label{eq:RBN2}
\end{align}
where $x\lt_{NR}y=NRx\lt y+x\lt NRy$ and $x\rt_{NR}y=NRx\rt y+x\rt NRy$.
\end{definition}

Since $(N,N)$ is a Nijenhuis pair with the regular representation, $(R,N)$ is an RBN-structure if and only if $(R,N,N)$ is a Kupershmidt-Nijenhuis structure on $\aa$ with respect to the regular representation. In particular, by Proposition~\ref{prop:47}, $NR$ is a Rota--Baxter operator compatible with $R$.

\begin{exam}\label{ex:RBN}
Fix $(\sigma,\tau)\in\kk^2\setminus\{(0,0)\}$ and let $\aa$ be the pre-Leibniz algebra with basis $\{e_1,e_2\}$ and nonzero products $e_1\lt e_1=\sigma e_2$, $e_1\rt e_1=\tau e_2$, i.e.\ $x\lt y=\sigma x_1y_1e_2$ and $x\rt y=\tau x_1y_1e_2$ (every term of \eqref{eq:PL1}--\eqref{eq:PL3} vanishes). When $\sigma+\tau=1$, $\aa$ splits the two-dimensional Leibniz algebra $[e_1,e_1]=e_2$. For a linear map $A$ on $\aa$ we write $A(e_j)=\sum_i a_{ij}e_i$. Direct computation on basis vectors gives the following.

(a) The Rota--Baxter operators are exactly
\begin{equation}
R_{\rm I}(e_1)=pe_2,\quad R_{\rm I}(e_2)=ce_2;\qquad R_{\rm II}(e_1)=2ce_1+pe_2,\quad R_{\rm II}(e_2)=ce_2\qquad(p,c\in\kk);
\end{equation}
indeed, writing $R=(r_{ij})$, the conditions reduce to $r_{12}=0$ and $r_{11}(r_{11}-2r_{22})=0$.

(b) The Nijenhuis operators are exactly $N(e_1)=\lambda e_1+\mu e_2$, $N(e_2)=\lambda e_2$ ($\lambda,\mu\in\kk$); indeed, writing $N=(n_{ij})$, \eqref{eq:Nij} reduces to $n_{12}=0$ and $(n_{11}-n_{22})^2=0$. Then $\aa_N$ has parameters $(\lambda\sigma,\lambda\tau)$.

(c) We have $NR(e_1)-RN(e_1)=\mp\mu c\,e_2$ for $R=R_{\rm I},R_{\rm II}$ and $NR(e_2)=RN(e_2)$, so \eqref{eq:RBN1} holds if and only if $\mu c=0$. By \eqref{eq:key1}, under \eqref{eq:RBN1}, \eqref{eq:RBN2a}--\eqref{eq:RBN2} are equivalent to $N(Rx\lt y+x\lt Ry)=Rx\lt Ny+Nx\lt Ry$ (and similarly for $\rt$), and both sides equal $2\sigma\lambda r_{11}x_1y_1e_2$. Hence $(R,N)$ is an RBN-structure if and only if $\mu c=0$.

(d) The operators $R_{\rm II}$ form the linear space spanned by $R_1=\mathrm{diag}(2,1)$ and $R_2$ with $R_2(e_1)=e_2$, $R_2(e_2)=0$; hence any two of them are compatible. Let $K_1=R_1$ and $K_2=\gamma R_1+\delta R_2$. Then $K_1$ is invertible, and the maps of Proposition~\ref{prop:48} are
\begin{equation}
S=K_1^{-1}K_2:\ e_1\mapsto\gamma e_1+\delta e_2,\ e_2\mapsto\gamma e_2;\qquad N=K_2K_1^{-1}:\ e_1\mapsto\gamma e_1+\tfrac{\delta}{2}e_2,\ e_2\mapsto\gamma e_2 .
\end{equation}
So $(R_1,S,N)$ and $(K_2,S,N)$ are Kupershmidt-dual-Nijenhuis structures, with $S\neq N$ when $\delta\neq0$. Since all actions of the regular representation are scalar multiples of $J\colon e_1\mapsto e_2,\ e_2\mapsto0$, which commutes with $S=\gamma\Id_\aa+\delta J$, the pair $(N,S)$ is perfect, so these are also Kupershmidt-Nijenhuis structures. For $\delta\neq0$, $(R_1,N)$ is not an RBN-structure by (c), since $\mu c=\delta/2\neq0$. Moreover, since $J^2=0$ and $R_1J=R_2$, the hierarchy of Proposition~\ref{thm:hierarchy} generated by $(R_1,S,N)$ is $K_k=R_1S^k=\gamma^kR_1+k\gamma^{k-1}\delta R_2$, $k\in\mathbb N$.
\end{exam}

To compare RBN-structures with $r$-$n$ structures, we need an invariant bilinear form.

\begin{definition}\label{def:quad}
A \emph{quadratic pre-Leibniz algebra} is a pre-Leibniz algebra $(\aa,\lt,\rt)$ together with a nondegenerate bilinear form $q$ such that the map $\varphi:\aa\to\aa^*$, $\pair{\varphi(x)}{y}=q(x,y)$, is an isomorphism from the regular representation to the dual regular representation $\aa^\vee$. Explicitly, for all $x,y,z\in\aa$,
\begin{equation}
\begin{aligned}
q(x\lt y,z)&=q(y,x\lt z), & q(x\rt y,z)&=-q(y,x\circ z),\\
q(y\lt x,z)&=q(y,x\circ z+z\circ x), & q(y\rt x,z)&=-q(y,x\lt z+z\rt x).
\end{aligned}
\end{equation}
We write $q^\sharp:=\varphi^{-1}:\aa^*\to\aa$.
\end{definition}

Equivalently, for all $x\in\aa$,
\begin{align}
\varphi L_\lt(x)&=-L_\lt^*(x)\varphi, & \varphi L_\rt(x)&=L_\circ^*(x)\varphi,\label{eq:quadiso}\\
\varphi R_\lt(x)&=-(L_\circ^*+R_\circ^*)(x)\varphi, & \varphi R_\rt(x)&=(L_\lt^*+R_\rt^*)(x)\varphi .\label{eq:quadiso4}
\end{align}

On a quadratic pre-Leibniz algebra, RBN-structures and pre-Leibniz $r$-$n$ structures correspond to each other. This is the pre-Leibniz analogue of the corresponding result for Leibniz algebras in \cite{SJ}.

\begin{thm}\label{thm:RBNrn}
Let $(\aa,q)$ be a quadratic pre-Leibniz algebra, $R:\aa\to\aa$ a linear map and $N:\aa\to\aa$ a Nijenhuis operator. Let $\pi\in\aa\otimes\aa$ be given by $\pi^\sharp=Rq^\sharp$, and suppose that
\begin{equation}\label{eq:qN}
q^\sharp N^{t}=Nq^\sharp,\quad\text{i.e.}\quad q(Nx,y)=q(x,Ny)\ \ \forall\,x,y\in\aa .
\end{equation}
Then $(R,N)$ is an RBN-structure on $\aa$ if and only if $(\pi,N)$ is a pre-Leibniz $r$-$n$ structure.
\end{thm}
\begin{proof}
Let $\varphi=(q^\sharp)^{-1}$, and let $\alpha=\varphi(x)$, $\beta=\varphi(y)$. Then $\pi^\sharp\alpha=Rx$ and $\pi^\sharp\beta=Ry$. By \eqref{eq:quadiso}--\eqref{eq:quadiso4}, we have
\begin{equation}\label{eq:quadR}
\begin{aligned}
\alpha\lt_{\pi^\sharp}\beta&=-L_\lt^*(Rx)\varphi(y)-(L_\circ^*+R_\circ^*)(Ry)\varphi(x)\\
&=\varphi\big(Rx\lt y+x\lt Ry\big)=\varphi(x\lt_Ry),
\end{aligned}
\end{equation}
and similarly $\alpha\rt_{\pi^\sharp}\beta=\varphi(x\rt_Ry)$. Thus $\pi^\sharp(\alpha\lt_{\pi^\sharp}\beta)=R(x\lt_Ry)$ and $\pi^\sharp\alpha\lt\pi^\sharp\beta=Rx\lt Ry$, and similarly for $\rt$, which implies that $\pi$ is a pre-Leibniz $r$-matrix if and only if $R$ is a Rota--Baxter operator. By \eqref{eq:qN}, $N^{t}\varphi=\varphi N$, so $\pi^\sharp N^{t}=RNq^\sharp$ and $N\pi^\sharp=NRq^\sharp$. Thus \eqref{eq:rn1} holds if and only if \eqref{eq:RBN1} holds. Moreover, by \eqref{eq:quadR} with $R$ replaced by $NR$, and by $N^{t}\varphi=\varphi N$, we have
\begin{align}
\alpha\lt_{N\pi^\sharp}\beta-\alpha\,(\lt_{\pi^\sharp})_{N^{t}}\,\beta&=\varphi(x\lt_{NR}y)-\big(N^{t}\alpha\lt_{\pi^\sharp}\beta+\alpha\lt_{\pi^\sharp}N^{t}\beta-N^{t}(\alpha\lt_{\pi^\sharp}\beta)\big)\notag\\\displaybreak[0]
&=\varphi(x\lt_{NR}y)-\big(\varphi(Nx)\lt_{\pi^\sharp}\varphi(y)+\varphi(x)\lt_{\pi^\sharp}\varphi(Ny)-N^{t}\varphi(x\lt_Ry)\big)\notag\\\displaybreak[0]
&=\varphi(x\lt_{NR}y)-\varphi\big(Nx\lt_Ry+x\lt_RNy-N(x\lt_Ry)\big)\notag\\\displaybreak[0]
&=\varphi\big(x\lt_{NR}y-x\,(\lt_R)_N\,y\big),
\end{align}
and similarly for $\rt$. Since $\varphi$ is bijective, \eqref{eq:rn2a}--\eqref{eq:rn2} hold if and only if \eqref{eq:RBN2a}--\eqref{eq:RBN2} hold. We finish the proof.
\end{proof}

Let $B$ be a bilinear form on $\aa$, and define $B^\flat:\aa\to\aa^*$ by $\pair{B^\flat(x)}{y}=B(x,y)$. We say that $B$ is \emph{nondegenerate} if $B^\flat$ is an isomorphism, and in that case we put $B^\sharp:=(B^\flat)^{-1}:\aa^*\to\aa$. We say that $B$ is \emph{closed} if, for all $x,y,z\in\aa$,
\begin{align}
B(x\lt y,z)&=B(y,x\lt z)+B(x,y\circ z+z\circ y),\label{eq:closed1}\\
B(x\rt y,z)&=-B(y,x\circ z)-B(x,y\lt z+z\rt y).\label{eq:closed2}
\end{align}
As in \cite{SJ}, closedness means that $B^\sharp$ is a Kupershmidt operator with respect to $\aa^\vee$. Indeed, let $\alpha,\beta\in\aa^*$ and put $x=B^\sharp\alpha$, $y=B^\sharp\beta$, so that $\pair{\alpha}{w}=B(x,w)$ and $\pair{\beta}{w}=B(y,w)$ for all $w\in\aa$. By \eqref{eq:dualreg}, for every $z\in\aa$,
\begin{align}
\pair{\alpha\lt_{B^\sharp}\beta}{z}&=\pair{\beta}{x\lt z}+\pair{\alpha}{y\circ z+z\circ y}=B(y,x\lt z)+B(x,y\circ z+z\circ y),\label{eq:closedcomp}\\
\pair{\alpha\rt_{B^\sharp}\beta}{z}&=-\pair{\beta}{x\circ z}-\pair{\alpha}{y\lt z+z\rt y}=-B(y,x\circ z)-B(x,y\lt z+z\rt y),\label{eq:closedcomp2}
\end{align}
while $\pair{B^\flat(x\lt y)}{z}=B(x\lt y,z)$ and $\pair{B^\flat(x\rt y)}{z}=B(x\rt y,z)$. Hence, if $B$ is nondegenerate, then $B$ is closed if and only if $B^\sharp$ is a Kupershmidt operator on $\aa$ with respect to $\aa^\vee$.

Combining closed bilinear forms with Nijenhuis operators, we obtain the following notion.

\begin{definition}\label{def:BN}
Let $N$ be a Nijenhuis operator on a pre-Leibniz algebra $(\aa,\lt,\rt)$ and $B$ a nondegenerate closed bilinear form. We say that $(B,N)$ is a \emph{BN-structure} on $\aa$ if
\begin{equation}\label{eq:BN}
B(Nx,y)=B(x,Ny),\qquad\forall\,x,y\in\aa,
\end{equation}
and the bilinear form $B_N(x,y):=B(x,Ny)$ also satisfies \eqref{eq:closed1}--\eqref{eq:closed2}; explicitly,
\begin{equation}
\begin{aligned}
B(x\lt y,Nz)&=B\big(y,N(x\lt z)\big)+B\big(x,N(y\circ z+z\circ y)\big),\\
B(x\rt y,Nz)&=-B\big(y,N(x\circ z)\big)-B\big(x,N(y\lt z+z\rt y)\big).
\end{aligned}
\end{equation}
\end{definition}

BN-structures provide further examples of Kupershmidt-dual-Nijenhuis structures.

\begin{thm}\label{thm:BN}
Let $N$ be a Nijenhuis operator on a pre-Leibniz algebra $\aa$ and $B$ a nondegenerate closed bilinear form. If $(B,N)$ is a BN-structure, then $(B^\sharp,N^{t},N)$ is a Kupershmidt-dual-Nijenhuis structure on $\aa$ with respect to the dual regular representation $\aa^\vee$.
\end{thm}
\begin{proof}
Since $B$ is closed, $B^\sharp$ is a Kupershmidt operator with respect to $\aa^\vee$. By \eqref{eq:BN}, we have $\pair{B^\flat NB^\sharp\alpha}{w}=B(B^\sharp\alpha,Nw)=\pair{N^{t}\alpha}{w}$, which implies that $NB^\sharp=B^\sharp N^{t}$. By Proposition~\ref{thm:dualpair}, $(N,N^{t})$ is a dual-Nijenhuis pair with respect to $\aa^\vee$. By \eqref{eq:key1}, we only need to prove \eqref{eq:F} for $K=B^\sharp$ and $S=N^{t}$. Let $x=B^\sharp\alpha$ and $y=B^\sharp\beta$. For all $z\in\aa$, by \eqref{eq:closedcomp} and \eqref{eq:closed1} for $B$ and for $B_N$, we have
\begin{align}
\pair{N^{t}(\alpha\lt_{B^\sharp}\beta)}{z}&=B(y,x\lt Nz)+B(x,y\circ Nz+Nz\circ y)=B(x\lt y,Nz)\notag\\\displaybreak[0]
&=B\big(y,N(x\lt z)\big)+B\big(x,N(y\circ z+z\circ y)\big)\notag\\\displaybreak[0]
&=\pair{N^{t}\beta}{x\lt z}+\pair{N^{t}\alpha}{y\circ z+z\circ y}\notag\\\displaybreak[0]
&=\pair{-L_\lt^*(x)N^{t}\beta-(L_\circ^*+R_\circ^*)(y)N^{t}\alpha}{z},
\end{align}
which is \eqref{eq:F} for $K=B^\sharp$ and $S=N^{t}$. Similarly, the $\rt$-part of \eqref{eq:F} follows from \eqref{eq:closedcomp2} and \eqref{eq:closed2}. We finish the proof.
\end{proof}

The following result is a consequence of Proposition~\ref{prop:47} and Theorem~\ref{thm:BN}.

\begin{corollary}\label{cor:BN}
Let $N$ be a Nijenhuis operator on a pre-Leibniz algebra $\aa$ and $B$ a nondegenerate closed bilinear form. If $(B,N)$ is a BN-structure, then $B^\sharp$ and $NB^\sharp=B^\sharp N^{t}$ are compatible Kupershmidt operators on $\aa$ with respect to $\aa^\vee$.
\end{corollary}

We illustrate Theorems~\ref{thm:RBNrn} and \ref{thm:BN} on the algebras of Example~\ref{ex:RBN}.

\begin{exam}\label{ex:quadRBN}
Let $\aa$ be as in Example~\ref{ex:RBN}, and write $q_{ij}=q(e_i,e_j)$, $B_{ij}=B(e_i,e_j)$.

(a) The invariance conditions of Definition~\ref{def:quad} reduce to $q_{22}=0$, $\sigma(q_{21}-q_{12})=0$, $(\sigma+\tau)q_{12}+\tau q_{21}=0$ and $\sigma q_{21}=2(\sigma+\tau)q_{12}$. Hence $\aa$ is quadratic if and only if $\sigma=-2\tau\neq0$, and then $q(e_1,e_1)=s$, $q(e_1,e_2)=q(e_2,e_1)=\theta$, $q(e_2,e_2)=0$ with $\theta\neq0$. Every Nijenhuis operator of Example~\ref{ex:RBN}(b) satisfies \eqref{eq:qN}, and, in the dual basis $\{e^1,e^2\}$, $q^\sharp(e^1)=\theta^{-1}e_2$, $q^\sharp(e^2)=\theta^{-1}e_1-s\theta^{-2}e_2$. For $R=R_{\rm II}$, $\pi^\sharp=Rq^\sharp$ gives
\begin{equation}
\pi=c\theta^{-1}\,e_1\otimes e_2+2c\theta^{-1}\,e_2\otimes e_1+(p\theta^{-1}-cs\theta^{-2})\,e_2\otimes e_2,
\end{equation}
which is not symmetric when $c\neq0$. By Theorem~\ref{thm:RBNrn} and Example~\ref{ex:RBN}(c), $\pi$ is a pre-Leibniz $r$-matrix, and $(\pi,N)$ is a pre-Leibniz $r$-$n$ structure if and only if $\mu c=0$.

(b) The closedness conditions reduce to $B_{22}=0$, $\sigma B_{21}=(3\sigma+2\tau)B_{12}$ and $\tau B_{21}=-2(\sigma+\tau)B_{12}$, a linear system of determinant $(2\sigma+\tau)(\sigma+2\tau)$. Hence a nondegenerate closed form exists if and only if $\tau=-2\sigma$ or $\sigma=-2\tau$, and then
\begin{equation}
B=\begin{pmatrix}s&\theta\\-\theta&0\end{pmatrix}\ \ (\tau=-2\sigma),\qquad B=\begin{pmatrix}s&\theta\\2\theta&0\end{pmatrix}\ \ (\sigma=-2\tau),\qquad \theta\neq0,
\end{equation}
which are in general neither symmetric nor skew-symmetric. Since $B(N(e_1),e_1)-B(e_1,N(e_1))=\mu(B_{21}-B_{12})$, the BN-structures are the pairs $(B,\lambda\Id_\aa)$, and $B^\sharp$, $\lambda B^\sharp$ are compatible by Corollary~\ref{cor:BN}.
\end{exam}

Applying Proposition~\ref{thm:hierarchy} to the structures of Theorems~\ref{thm:rnON} and~\ref{thm:BN} and to RBN-structures, we obtain the following hierarchies.

\begin{corollary}\label{cor:hierarchies}
Let $\aa$ be a pre-Leibniz algebra and $N$ a Nijenhuis operator on $\aa$.
\begin{enumerate}[(i)]
\item If $(R,N)$ is an RBN-structure, then, for every $k\in\mathbb N$, $N^kR$ is a Rota--Baxter operator, $(N^kR,N)$ is an RBN-structure, and the operators $N^kR$ are pairwise compatible.
\item If $(\pi,N)$ is a pre-Leibniz $r$-$n$ structure, let $\pi_k\in\aa\otimes\aa$ be given by $\pi_k^\sharp=N^k\pi^\sharp=\pi^\sharp(N^{t})^k$. Then each $\pi_k$ is a pre-Leibniz $r$-matrix, $(\pi_k,N)$ is a pre-Leibniz $r$-$n$ structure, and the operators $\pi_k^\sharp$ are pairwise compatible Kupershmidt operators with respect to $\aa^\vee$.
\item If $(B,N)$ is a BN-structure, then the operators $N^kB^\sharp=B^\sharp(N^{t})^k$, $k\in\mathbb N$, are pairwise compatible Kupershmidt operators on $\aa$ with respect to $\aa^\vee$.
\end{enumerate}
\end{corollary}
\begin{proof}
By Theorems~\ref{thm:rnON} and~\ref{thm:BN} and the discussion after Definition~\ref{def:RBN}, $(R,N,N)$, $(\pi^\sharp,N^{t},N)$ and $(B^\sharp,N^{t},N)$ are Kupershmidt-(dual-)Nijenhuis structures. The conclusion follows from Proposition~\ref{thm:hierarchy}.
\end{proof}


\begin{thebibliography}{99}
\small\linespread{1.2}\selectfont
\setlength{\itemsep}{-3.5pt}

\bibitem{Aguiar} M. Aguiar, Pre-Poisson algebras, \emph{Lett. Math. Phys.} \textbf{54} (2000), 263--277.

\bibitem{B1} C. Bai, A unified algebraic approach to the classical Yang--Baxter equation, \emph{J. Phys. A: Math. Theor.} \textbf{40} (2007), 11073--11082.

\bibitem{BBGN} C. Bai, O. Bellier, L. Guo, X. Ni, Splitting of operations, Manin products, and Rota--Baxter operators, \emph{Int. Math. Res. Not.} \textbf{2013} (2013), 485--524.

\bibitem{BGN} C. Bai, L. Guo, X. Ni, Nonabelian generalized Lax pairs, the classical Yang--Baxter equation and PostLie algebras, \emph{Comm. Math. Phys.} \textbf{297} (2010), 553--596.

\bibitem{Bal} D. Balavoine, Deformation of algebras over a quadratic operad, \emph{Contemp. Math.} \textbf{202} (1997), 207--234.

\bibitem{Ba} G. Baxter, An analytic problem whose solution follows from a simple algebraic identity, \emph{Pacific J. Math.} \textbf{10} (1960), 731--742.

\bibitem{BenabdelhafidhNijenhuisMorphisms} S. Benabdelhafidh, Deformation theory and minimal models of operads for Nijenhuis algebra morphisms, \emph{J. Algebra Appl.}, to appear, doi:10.1142/S0219498827501350.

\bibitem{Bloh} A. M. Bloh, On a generalization of the concept of Lie algebra, \emph{Dokl. Akad. Nauk SSSR} \textbf{165} (1965), 471--473.

\bibitem{BW} M. Bordemann, F. Wagemann, Global integration of Leibniz algebras, \emph{J. Lie Theory} \textbf{27} (2017), 555--567.

\bibitem{QRZ} Y. Chen, L. Qin, S. Ren, R. Zhang, Kupershmidt--Nijenhuis structures on pre-Malcev algebras, \emph{Bull. Malays. Math. Sci. Soc.} \textbf{49} (2026), no.~1, Paper No.~17, doi:10.1007/s40840-025-02012-2.

\bibitem{CK} A. Connes, D. Kreimer, Renormalization in quantum field theory and the Riemann--Hilbert problem I. The Hopf algebra structure of graphs and the main theorem, \emph{Comm. Math. Phys.} \textbf{210} (2000), 249--273.

\bibitem{Das} A. Das, Pre-Leibniz algebras, \emph{Comm. Algebra} \textbf{52} (2024), no.~8, 3383--3399, doi:10.1080/00927872.2024.2318492.

\bibitem{DW} B. Dherin, F. Wagemann, Deformation quantization of Leibniz algebras, \emph{Adv. Math.} \textbf{270} (2015), 21--48.

\bibitem{Dor} I. Ya. Dorfman, \emph{Dirac Structures and Integrability of Nonlinear Evolution Equations}, John Wiley, 1993.

\bibitem{FM} A. Fialowski, A. Mandal, Leibniz algebra deformations of a Lie algebra, \emph{J. Math. Phys.} \textbf{49} (2008), 093511.

\bibitem{FF} B. Fuchssteiner, A. S. Fokas, Symplectic structures, their B\"acklund transformations and hereditary symmetries, \emph{Phys. D} \textbf{4} (1981/82), 47--66.

\bibitem{Guo} L. Guo, \emph{An Introduction to Rota--Baxter Algebra}, Surveys of Modern Mathematics 4, International Press, Somerville, MA; Higher Education Press, Beijing, 2012.

\bibitem{HarrathiMabroukNcibSilvestrov} F. Harrathi, S. Mabrouk, O. Ncib, S. Silvestrov, Kupershmidt operators on Hom-Malcev algebras and their deformation, \emph{Int. J. Geom. Methods Mod. Phys.} \textbf{20} (2023), 2350046.

\bibitem{HLS} Y. Hu, J. Liu, Y. Sheng, Kupershmidt-(dual-)Nijenhuis structures on a Lie algebra with a representation, \emph{J. Math. Phys.} \textbf{59} (2018), 081702.

\bibitem{KW} M. Kinyon, A. Weinstein, Leibniz algebras, Courant algebroids, and multiplications on reductive homogeneous spaces, \emph{Amer. J. Math.} \textbf{123} (2001), 525--550.

\bibitem{KSM} Y. Kosmann-Schwarzbach, F. Magri, Poisson--Nijenhuis structures, \emph{Ann. Inst. H. Poincar\'e Phys. Th\'eor.} \textbf{53} (1990), 35--81.

\bibitem{KR} Y. Kosmann-Schwarzbach, V. Rubtsov, Compatible structures on Lie algebroids and Monge--Amp\`ere operators, \emph{Acta Appl. Math.} \textbf{109} (2010), 101--135.

\bibitem{Ku} B. A. Kupershmidt, What a classical $r$-matrix really is, \emph{J. Nonlinear Math. Phys.} \textbf{6} (1999), 448--488.

\bibitem{LBS} J. Liu, C. Bai, Y. Sheng, Compatible $O$-operators on bimodules over associative algebras, \emph{J. Algebra} \textbf{532} (2019), 80--118.

\bibitem{Lod} J.-L. Loday, Une version non commutative des alg\`ebres de Lie : les alg\`ebres de Leibniz, \emph{Enseign. Math.} \textbf{39} (1993), 269--293.

\bibitem{LP} J.-L. Loday, T. Pirashvili, Universal enveloping algebras of Leibniz algebras and (co)homology, \emph{Math. Ann.} \textbf{296} (1993), 139--158.

\bibitem{Mabrouk} S. Mabrouk, Deformation of Kupershmidt operators and Kupershmidt-Nijenhuis structures of a Malcev algebra, \emph{Hacet. J. Math. Stat.} \textbf{51} (2022), 199--217.

\bibitem{MM} F. Magri, C. Morosi, A geometrical characterization of integrable Hamiltonian systems through the theory of Poisson--Nijenhuis manifolds, Quaderno S/19, Milan, 1984.

\bibitem{NR} A. Nijenhuis, R. W. Richardson, Deformation of Lie algebra structures, \emph{J. Math. Mech.} \textbf{17} (1967), 89--105.

\bibitem{RRH} Z. Ravanpak, A. Rezaei-Aghdam, G. Haghighatdoost, Invariant Poisson--Nijenhuis structures on Lie groups and classification, \emph{Int. J. Geom. Methods Mod. Phys.} \textbf{15} (2018), 1850059.

\bibitem{STS} M. A. Semenov-Tian-Shansky, What is a classical $r$-matrix?, \emph{Funct. Anal. Appl.} \textbf{17} (1983), 259--272.

\bibitem{SJ} Q. Sun, N. Jing, $O$-operators and related structures on Leibniz algebras, \emph{Comm. Algebra} \textbf{51} (2023), no.~5, 2199--2216.

\bibitem{ST} R. Tang, Y. Sheng, Leibniz bialgebras, relative Rota--Baxter operators, and the classical Leibniz Yang--Baxter equation, \emph{J. Noncommut. Geom.} \textbf{16} (2022), 1179--1211.

\bibitem{U} K. Uchino, Quantum analogy of Poisson geometry, related dendriform algebras and Rota--Baxter operators, \emph{Lett. Math. Phys.} \textbf{85} (2008), 91--109.

\bibitem{WSBL} Q. Wang, Y. Sheng, C. Bai, J. Liu, Nijenhuis operators on pre-Lie algebras, \emph{Commun. Contemp. Math.} \textbf{21} (2019), 1850050.
\end{thebibliography}
\end{document}